\documentclass{amsart}

\newtheorem{theorem}{Theorem}[section]

\newtheorem{proposition}[theorem]{Proposition}
\newtheorem{corollary}[theorem]{Corollary}
\newtheorem{question}[theorem]{Question}

\theoremstyle{definition}

\newtheorem{remark}[theorem]{Remark}

\usepackage{amscd,amssymb}
\usepackage[mathscr]{eucal}
\usepackage{enumerate}
\usepackage{amsmath}
\usepackage {xcolor}

\begin{document}

\title[Nowicki's conjecture]
{On Nowicki's conjecture: a survey and a new result}
\author[Lucio Centrone, Andre Dushimirimana, {\c S}ehmus F{\i}nd{\i}k]
{Lucio Centrone, Andre Dushimirimana, {\c S}ehmus F{\i}nd{\i}k}
\address{Dipartimento di Matematica, Università degli Studi di Bari "Aldo Moro", via E. Orabona 4, 70125, Bari, Italy}
\address{IMECC, Universidade Estadual de Campinas,
Rua Sergio Buarque de Holanda 651, 13083-859,
Campinas (SP), Brazil}
\email{centrone@unicamp.br}
\address{Department of Mathematics,
\c{C}ukurova University, 01330 Balcal\i,
 Adana, Turkey}
\email{duandosh@yahoo.fr}
\email{sfindik@cu.edu.tr}

\thanks
{The second named author is partially supported by T\"UB\.ITAK}


\begin{abstract}
The goal of the paper is twofold: it aims to give an extensive set of tools and bibliography towards Nowicki's conjecture both in an associative setting; it establishes a new result about Nowicki's conjecture for the free metabelian Poisson algebra. 
\end{abstract}

\maketitle

\section{Introduction}
In this paper we will handle Nowicki's conjecture from its classical formulation to to some related problems. Our intent is producing a survey toward this theme and adding a new brick concerning Nowicki's conjecture in a Poisson algebra setting.

Our first aim is introducing the so-called Weitzenb\"ock derivations. We start considering a field $K$ of characteristic zero and a finite set of variables $X_m=\{x_1,\ldots,x_m\}$; we denote by $KX_m$ the vector space over $K$ with basis $X_m$ and we shall denote by $K[X_m]$ the (commutative) unitary polynomial algebra generated by $X_m$. Let $\delta$ be a non-zero nilpotent linear map of $KX_m$ that can be extended to a derivation of $K[X_m]$, then $\delta$ is called \textit{Weitzenb\"ock}. This name is due to the classical result of Weitzenb\"ock \cite{W}, dating back to 1932, which states that the algebra $K[X_m]^{\delta}=\text{ker}(\delta)$ of constants of the derivation $\delta$ in the algebra $K[X_m]$ is
finitely generated.
If $\delta$ is a Weitzenb\"ock derivation, then it is locally nilpotent and the linear map
$\text{exp}(\delta)$ acting on the vector space $KX_m$
is unipotent. Hence the algebra $K[X_m]^{\delta}$ of constants of $\delta$ is equal to the algebra of invariants 
\[K[X_m]^{\text{exp}(\delta)}=\{x\in K[X_m]\mid \text{exp}(\delta)(x)=x\}.\]
The algebra of invariants of $\text{exp}(\delta)$ is equal to the algebra of invariants $K[X_m]^{G}$, where the group $G$ consists of all elements $\text{exp}(c\delta)$, $c\in K$.
As an abstract group, the group $G$ is isomorphic to the unitriangular group consisting of all $2\times 2$ upper triangular matrices with 1's on the diagonal.
The group $G$ can be embedded into the general linear group $GL_m(K)$ via the representation corresponding
to the Jordan normal form of $\delta$.
Thus we are allowed to study the algebra of constants $K[X_m]^{\delta}$ with the methods of the classical invariant theory.
The computational aspects of algebras of constants and invariant theory are depicted in part in the books by Nowicki \cite{N}, Derksen and Kemper \cite{DK},
and Sturmfels \cite{St}.

We can say more about Weitzenb\"ock derivations. For instance, take a Weitzenb\"ock derivation, then its Jordan normal form consists of Jordan cells with zero diagonals.
Certainly, Weitzenb\"ock derivations are in one-to-one correspondence with the partitions of $m$, up to a linear change of variables. This means there are essentially finite number of  Weitzenb\"ock derivations for a fixed dimension $m$.
In particular, let $m=2d$, $d\geq1$, and assume that the Jordan form of $\delta$ contains
the Jordan cells of size $2\times2$ only, i.e.,
\[
J(\delta)=\left(\begin{matrix}
0&1&\cdots&0&0\\
0&0&\cdots&0&0\\
\vdots&\vdots&\ddots&\vdots&\vdots\\
0&0&\cdots&0&1\\
0&0&\cdots&0&0
\end{matrix}\right).
\]
We assume that $\delta(x_{2j})=x_{2j-1}$, $\delta(x_{2j-1})=0$, $j=1,\ldots,d$.
Nowicki conjectured in his book \cite{N} (Section 6.9, page 76) that in this case the algebra $K[X_{2d}]^{\delta}$
is generated by $x_1,x_3.\ldots,x_{2d-1}$, and $x_{2i-1}x_{2j}-x_{2i}x_{2j-1}$, $1\leq i<j\leq d$. The last conjecture will be called throughout the text \textit{classical Nowicki's conjecture}.

Nowicki's conjecture was proved by several authors with different techniques.
The first published proof appeared in 2004 by Khoury \cite{K1} in his PhD thesis,
followed by his paper \cite{K2}, where he makes use of a computational approach involving Gr\"obner basis techniques. Derksen and Panyushev applied ideas of classical invariant
theory in order to prove the Nowicki conjecture but their proof remained unpublished. Later in 2009, 
Drensky and Makar-Limanov \cite{DML} confirmed the conjecture
by an elementary proof from undergraduate algebra, without involving invariant theory. In this paper we shall highlight the key points of the proof given by Drensky and Makar-Limanov. Moreover, another simple short proof was given by Kuroda \cite{Kuroda} in 2009, too.
Bedratyuk \cite{Bed} proved the Nowicki conjecture by reducing it to a well known problem of classical invariant theory. We were able to find three more proofs of Nowicki's conjecture: two of them are cited in the paper by Kuroda \cite{Kuroda}, and the latest proof was given by Drensky in \cite{D2}.

In this paper we would like to attempt to give a complete account on the state of the art of Nowicki's conjecture presenting a sketch of the proof obtained by Drensky and Makar-Limanov. We will also treat a bit a generalization of Nowicki's conjecture (see Sections 3 and 4 in the text). We will also focus on recent developments of the following problem, strictly related to the classical Nowicki's conjecture, but with a scent of theory of algebras with polynomial identities (PI-algebras). Let us explain briefly our framework: let $\delta$ be a Weitzenb\"ock derivation of the free algebra $F_{m}(\mathfrak V)$ of rank $m$ in a given variety of algebras $\mathfrak V$ (or relatively free algebra of rank $m$ of $\mathfrak{V}$), then \[\textit{give an explicit set of generators of the algebra of constants $F_{m}^{\delta}(\mathfrak V)$.}\]
In \cite{DG} Drensky and Gupta studied Weitzenb\"ock derivations $\delta$ acting on $F_m(\mathfrak{V})$ proving that if the algebra $UT_2(K)$ of $2\times 2$ upper triangular matrices over a field $K$ of characteristic zero belongs to a variety $\mathfrak{V}$, then $F_m^\delta(\mathfrak{V})$ is not finitely generated whereas if $UT_2(K)$ does not belong to $\mathfrak{V}$, then $F_m^\delta(\mathfrak{V})$ is finitely generated by a result of Drensky in \cite{D1}.
Recently, Dangovski et al. \cite{DDF,DDF1} gave some new results in this direction. They showed that the algebra of constants $(F_d'(\mathfrak{V}))^\delta$ in the commutator ideal $F_d'(\mathfrak{V})$ of $F_d(\mathfrak{V})$
is a finitely generated module of $K[X_d]^{\delta}$ and $K[U_d,V_d]^{\delta}$, respectively.
Here $\delta$ acts on $U_d$ and $V_d$ in the same way as on $X_d$. In \cite{CF} the authors give an explicit set of generators of the algebra of constants of the variety generated by the infinite dimensional Grassmann algebra and of the free metabelian associative algebra. A good part of those result is presented explicitly in the text (Sections 5 and 6).

In the second part of the paper (Section 7) we will show a substantially new result. We consider the free metabelian Poisson algebra $P_{2n}$ of rank $2n$ over the field $K$ and we give out an explicit set of generators of its algebra of constants under the action of a Weitzenb\"ock derivation (see Theorems \ref{KB_3} and \ref{KB_4} in the text) as a vector space. 

We would like to stress the fact that in the study of varieties of (not necessarily associative) algebras over a field of characteristic zero, the metabelian identities play a crucial role. It is well known the identities of $UT_2(K)$ follow from the metabelian identity of order 4, so every variety $\mathfrak{V}$ contains $UT_2(K)$ or satisfies the Engel identity $[x_2,x_1,\ldots,x_1]\equiv0$.

 \section{Preliminaries}
 Let $K$ be a field of characteristic 0. 
We consider the commutative and unitary polynomial algebra $K[X_m]$ with generating set $X_m=\{x_1,\ldots,x_m\}$, $m\geq2$ over $K$. Of course, $K[X_m]$ is the free algebra of rank
$m$ in the variety of unitary commutative algebras over $K$. Now we consider the general linear group $GL_m(K)$ and $H\leq GL_m(K)$ a subgroup. We say a polynomial $p\in K[X_m]$ is $H$-\textit{invariant} if it is preserved under the action of each element of $H$.  The vector space $K[X_m]^H$ of all $H$-invariants turns out to be an algebra that is called the algebra of $H$-invariants.

The question whether the algebra $K[X_n]^H$ of invariants is finitely generated for every subgroup $H$ of $GL_m(K)$ is a special case of the \textbf{Hilbert's fourteenth problem} suggested by the German mathematician David Hilbert in 1900 at the International Congress of Mathematicians in Paris.

Although the answer to Hilbert's question turned out to be negative in general, as showed by Nagata in 1958, some remarkable affirmative cases have been handled as well.
Among them is the approach of Weitzenb\"ock, where he considered the \textbf{locally nilpotent} linear derivations $\delta$ of the algebra $K[X_m]$.
Then the kernel $\text{\rm ker}(\delta)=K[X_m]^{\delta}$ of the derivation $\delta$ is an algebra called the \textit{algebra of constants} of $K[X_m]$. He showed that the algebra $K[X_m]^{\delta}$ is finitely generated as an algebra and is equal to the algebra $K[X_m]^{UT_2(K)}$ of invariants of 
the unitriangular group \[{UT_2(K)}=\{\text{exp}(c\delta)\mid c\in K\},\] as already said in the introduction.

The next question is \[\textit{What is the explicit form of the generators of $K[X_n]^\delta$?}\] 
Nowicki conjectured in his book that the algebra $K[X_m,Y_m]^{\delta}$
is generated by \[\text{$x_1,\ldots,x_m$ and $u_{i,j}:=x_iy_j-x_jy_i$, $1\leq i<j\leq m$,}\]
assuming that $\delta(y_{i})=x_{i}$, $\delta(x_{i})=0$, $i=1,\ldots,m$, where $K[X_m,Y_m]$ denotes the commutative algebra generated by the set $X_m\cup Y_m$.

This is what we shall call \textit{Nowicki's conjecture} until the end of the paper. 
 
Let us introduce now the noncommutative setting in which we shall generalize Nowicki's conjecture. Let $X=\{x_1,x_2,\ldots\}$ be an infinite set of variables and let $\mathcal{C}$ denote a special class of algebras (for instance, associative algebras, Lie algebras, commutative algebras, etc.). We can consider the free $\mathcal{C}$-algebra $\mathcal{C}\{X\}$ generated by $X$ over $K$ and we shall call its elements \textit{$\mathcal{C}$-polynomials} or simply \textit{polynomials}. A polynomial $f=f(x_1,\ldots,x_n)$ is said to be a \textit{polynomial identity} for the algebra $A\in\mathcal{C}$ if $f(a_1,\ldots,a_n)=0$ for any $a_i\in A$. 

The set of all polynomial identities of a given algebra $A$ is denoted by $T_{\mathcal{C}}(A)$ (or $Id_{\mathcal{C}}(A)$) and is invariant under all $\mathcal{C}$-endomorphisms of $\mathcal{C}\{X\}$, i.e, it is a $T_{\mathcal{C}}$-ideal called the \textit{$T_{\mathcal{C}}$-ideal} of $A$. The algebra $\mathcal{C}\{X\}/Id_{\mathcal{C}}(A)$ is called \textit{relatively free $\mathcal{C}$-algebra} of $A$.


We can generalize Nowicki's conjecture to relatively free algebras in the following way:  let us consider the free algebra $\mathcal{C}\{X_m,Y_m\}$ endowed with the derivation $\delta$ such that $\delta(x_i)=0$ and $\delta(y_i)=x_i$ for any $i$.  Chosen $A\in\mathcal{C}$, $Id(A)$ is, of course, $\delta$-invariant, then $\delta$ induces a derivation $\delta$ on the relatively free algebra $\mathcal{C}_n(A):=\mathcal{C}\{X_m,Y_m\}/Id(A)$. 
\[\textit{Is the algebra of constants of $\mathcal{C}_n(A)$ finitely generated?}\]\[\textit{ Are the generators uniformly looking?}\]

    Let $A\in\mathcal{C}$ and consider the class \[\mathfrak{V}(A)=\{B\in \mathcal{C}|Id(A)\subseteq Id(B)\}\] that is called \textit{variety of $\mathcal{C}$-algebras generated by $A$}.
    Here is the answer to the first question: 
\begin{theorem}[Drensky \cite{D1}, Drensky and Gupta \cite{DG}]
Let $\mathfrak{V}:=\mathfrak{V}(A)$ be a variety of associative or Lie algebras, where $A$ is endowed with the derivation $\delta$. Then the algebra of constants of $A$ is not finitely generated if and only if $\mathfrak{V}$ contains $UT_2(K)$, the algebra of upper triangular $2\times 2$ matrices with entries from $K$ . 
\end{theorem}

\section{Drensky and Makar-Limanov's approach}
Herein we shall depict sketches of the proof of Nowicki's conjecture as showed by Drensky and Makar-Limanov in \cite{DML}.

Let $X_m=\{x_1,\ldots,x_m\}$, $Y_m=\{y_1,\ldots,y_m\}$ be two disjoint sets of commutative variables.
    Let $X'=\{x_1,\ldots,x_{m-1}\}$, $Y'=\{y_1,\ldots,y_{m-1}\}$, $U'=\{u_{i,j}|1\leq i,j\leq m-1\}$   (where $u_{i,j}=x_iy_j-x_jy_i$), $X=X_m$, $Y=Y_m$ and consider a constant $f=f(X',Y',x_m,y_m)$. Of course $f$ can be taken to be homogeneous in $(X,Y)$ and we carry on by induction on $m$ and on the total degree in $x_m,y_m$.
    
    \begin{remark}
    For a monomial $v\in K[X,Y]$ of degree $(d_1,d_2)$ with respect to $(X,Y)$, $\delta(v)$ is of degree $(d_1+1,d_2-1)$ with respect to $(X,Y)$.
    \end{remark}
    
    If $\underline{n=1}$, then by the previous remark, we get $K[x_1,y_1]^\delta=K[x_1]$ and we are done because in this case $u_{1,1}=0$.

We consider \underline{$m>1$} and we may suppose $\deg_{x_m,y_m}f>0$. We write \[f=(a_py_m^p+a_{p-1}x_my_m^{p-1}+\cdots+a_1x_m^{p-1}y_m+a_0x_m^{p})x_m^q,\] where $a_j\in K[X',Y']$, $a_p\neq0$.

\begin{remark}
Since both $f$ and $x_m^q$ are constants, the same holds for $f/x_m^q$. Hence we may assume $q=0$.
\end{remark}

\begin{remark}
Because \[0=\delta(f)\]\[=\delta(a_p)y_m^p+(pa_p+\delta(a_{p-1}))x_my_m^{p-1}+\cdots+(a_1+\delta(a_0))x_m^p\] we have $a_p$ is a constant and, by induction, has the form \[a_p=\sum x_{s_1}\cdots x_{s_k}b_s(U').\]
\end{remark}

It is easy to note that $\deg_X(f)\geq p+\deg_{Y'}(a_p)$, then, by the last remark, we may write $f$ as \[f=\sum x_{s_1}\cdots x_{s_p}c_s(X',U')y_m^p+\sum_{i=0}^{p-1}a_i(X',Y')x_m^{p-i}y_m^i.\]

\begin{remark}
This is very important: \[x_{s_j}y_m=(x_{s_j}y_m-x_my_{s_j})+x_my_{s_j}=u_{s_j,m}+x_my_{s_j}.\]
\end{remark}
    
    Then we get \[f=\sum c_s(X',U')\prod_{j=1}^pu_{s_j,m}+x_mf_1(X',Y',x_m,y_m)\] and $\deg_{x_m,y_m}(f_1)<\deg_{x_m,y_m}(f)$ but $f$, $x_m$, $c_s(X',U')\prod_{j=1}^pu_{s_j,m}$ are all constants, then $f_1$ is a constant and we are done by induction.

\section{Generalized Nowicki's conjecture}

    Let us take an integral domain $D$ over $K$, a set of variables $Y_m:=\{y_1,\ldots,y_m\}$ and the algebra $A:=D[Y_m]$. A derivation $\Delta$ on $A$ is said \textit{elementary} if $\Delta(D)=0$ and $\Delta(y_i)\in D$ for any $i$. \ We consider also the following objects  \[u_{i,j}:=\left|\begin{array}{cc}
       \Delta(y_i)  & y_j \\
       \Delta(y_j)  & y_i
    \end{array}\right|.\]
They are called \textit{determinants} and they belong to $D[y_m]^\Delta$. 

\begin{question}[Generalized Nowicki's conjecture]
Is the $D$-algebra $D[Y_m]^\Delta$ generated by the determinants?
\end{question}

    Of course the previous question is precisely the Nowicki's conjecture when $\Delta$ is such that $\Delta(D)=0$ and $\Delta(y_i)=x_i$ for any $i$.
    \ In the first proof of Nowicki's conjecture by Khoury, the author considered an elementary derivation $\Delta$ sending $\Delta(y_i)=x_i^{m_i}$. Then the algebra of constants $K[X_m,Y_m]^\Delta$ is generated by $X_m$ and the $u_{i,j}$'s.
    The more general result in this direction is due to Kuroda:
    \begin{theorem}
    Let $\Delta$ be an elementary derivation of $D[Y_m]$ such that $\Delta(y_i)=f_i$, for any $i$. Suppose further the $f_i$'s are algebraically independent over $K$. Then $D[Y_m]^\Delta$ is generated by the determinants if $D$ is \textbf{flat} over $K[f_1,\ldots,f_m]$.
    \end{theorem}

  Notice that this result is deeply different from the original version of Nowicki's conjecture.\ \textit{Nowicki's conjecture can be restated in the language of \textbf{classical invariant theory}} but Kuroda's result in case the $f_i$'s are not linear! 
  
  Anyway, using Kuroda's result, Drensky \cite{D2000} proved the next:
  
  \begin{theorem}
  Let $\Delta$ be an elementary derivation of $D[Y_m]$ such that $\Delta(y_i)=f_i=f_i(x_i)$, for any $i$, where the $f_i$'s are non-constant polynomials. Then $D[Y_m]^\Delta$ is generated by $X_m$ and the determinants.
  \end{theorem}
  
  Notice also that in all of the result above, the determinants generating the algebra of constants are not free generators. A prominent part of those works deals with finding a free set of generators.

\section{The free metabelian associative algebras}

We will focus on another generalization of Nowicki's conjecture. The environment now is that of relatively free algebras. The complete proofs of the results presented in the next two section can be found in the paper \cite{CF}.

We start off our investigation with the relatively free algebra of the associative metabelian algebra.
Let $K$ be a field of characteristic zero, $P_{2d}$ be the free unitary associative algebra of rank ${2d}$
over $K$, $P_{2d}'=P_{2d}[P_{2d},P_{2d}]P_{2d}$ be its commutator ideal generated by all elements of the form
\[
[x,y]=xy-yx, \quad x,y\in P_{2d}.
\]
Let us consider the quotient
algebra $F_{2d}=P_{2d}/(P_{2d}')^2$. The algebra $F_{2d}$ is the free algebra of rank $2d$ in the
variety of all associative algebras satisfying the polynomial identity $[x,y][z,t]\equiv0$.
Let the free associative metabelian algebra $F_{2d}$ be generated by $X_{2d}=\{x_1,\ldots,x_{2d}\}$.
We assume that all Lie commutators are left normed; i.e., 
\[
[x,y,z]=[[x,y],z], \quad x,y,z\in F_{2d}.
\]
The commutator ideal $F_{2d}'$ of $F_{2d}$ 
has the following basis (see e.g. \cite{B, D})
\[
x_1^{\xi_1}\cdots x_{2d}^{\xi_{2d}}[x_{i},x_{j},x_{j_1},\ldots,x_{j_m}],\quad \xi_i\geq 0, \quad i>j\leq j_1\leq\cdots\leq j_m\leq 2d.
\]
As a consequence of the metabelian identity in $F_{2d}'$ we have the following identity:
\[
x_{i_{\pi(1)}}\cdots x_{i_{\pi(m)}}[x_i,x_j,x_{j_{\sigma(1)}},\ldots,x_{j_{\sigma(n)}}]
\equiv x_{i_1}\cdots x_{i_m}[x_{i},x_{j},x_{j_1},\ldots,x_{j_n}],
\]
for any permutation $\pi\in S_m$ and $\sigma\in S_n$. Thus the commutator ideal 
$F_{2d}'$ can be ``seen'' as a module of polynomial algebra from both sides via the associative (left side) and Lie (right side) multiplication.

We recall now some of the results and constructions given in \cite{DDF1}. Let
$U_{2d}=\{u_1,\ldots,u_{2d}\}$ and $V_{2d}=\{v_1,\ldots,v_{2d}\}$ be two sets of commuting variables
and let $K[U_{2d},V_{2d}]$ be the polynomial algebra acting on $F_{2d}'$ as follows.
If $f\in F_{2d}'$, then
\[
fu_i=x_if, \quad fv_i=[f,x_i], \quad i=1,\ldots,2d.
\]
This action defines a $K[U_d,V_d]$-module structure on the vector space $F_{2d}'$.

We are going to construct a wreath product which is the same as the one used in \cite{DDF1}. It is a particular case
of the construction of Lewin \cite{L} given in \cite{DG} and is similar to the construction of
Shmel'kin \cite{Sh} in the case of free metabelian Lie algebras as appeared in \cite{DDF}.

Let $Y_{2d} = \{y_1, \ldots, y_{2d}\}$ and $V_{2d}'=\{v_1',\ldots,v_{2d}'\}$ be sets of commuting variables and let
$A_{2d}=\{a_1,\ldots,a_{2d}\}$ in $2d$ variables.
Now let $M_{2d}$ be the free $K[U_{2d},V_{2d}']$-module generated by $A_{2d}$
equipped with with trivial multiplication
$M_{2d}\cdot M_{2d}=0$. We endow $M_{2d}$ with a structure of a free $K[Y_{2d}]$-bimodule
structure via the action
\[
y_ja_i = a_iu_j, \quad a_iy_j = a_iv_j', \quad i,j = 1, \ldots, 2d.
\]
The wreath product $W_{2d}=K[Y_{2d}]\rightthreetimes M_{2d}$ 
is an algebra satisfying the metabelian identity. As well as in \cite{L} $F_{2d}$ can be embedded into $W_{2d}$. In fact we have the following result.

\begin{proposition}\label{embedding}
The mapping $\varepsilon :x_j\to y_j+a_j$, $j=1,\ldots,2d$, extends to an embedding $\varepsilon$ of $F_{2d}$ into $W_{2d}$.
\end{proposition}

We identify $v_i=v_i'-u_i$, $i=1,\ldots,2d$, and get
\begin{equation}\label{commutator element}
\varepsilon(x_{i_1}\cdots x_{i_m}[x_{i},x_{j},x_{j_1},\ldots,x_{j_n}])
=(a_iv_j-a_jv_i)v_{j_1}\cdots v_{j_n}u_{i_1}\cdots u_{i_m}
\end{equation}
Thus we may assume that $M_{2d}$ is a free $K[U_{2d},V_{2d}]$-module. Clearly the commutator ideal
$F_{2d}'$ is embedded into $M_{2d}$, too.
An element $\sum a_if(U_{2d},V_{2d})\in M_{2d}$
is an image of some element from $F_{2d}'$ if and only if $\sum v_if(U_{2d},V_{2d})=0$,
as a consequence of (\ref{commutator element}).

Let $\delta$ be the Weitzenb\"ock derivation of $F_{2d}$ acting on the
variables $U_{2d},V_{2d}$, as well as explained in the Introduction on $X_{2d}$. By \cite{DDF1} we know that
the vector space $M_{2d}^{\delta}$ of the constants of $\delta$ in the $K[U_{2d},V_{2d}]$-module $M_{2d}$
is a $K[U_{2d},V_{2d}]^{\delta}$-module.
The following results are particular cases of \cite{D1} (Proposition 3) and \cite{DDF1}.

\begin{theorem}\label{theorem for finite generation}
The vector spaces $(F_{2d}')^{\delta}$ and $M_{2d}^{\delta}$ are finitely generated $K[U_{2d},V_{2d}]^{\delta}$-modules.
\end{theorem}

In the sequel we shall write down an explicit set of generators for the algebra of constants $K[U_{2d},V_{2d}]^{\delta}$ as a consequence of
the Nowicki conjecture \cite{N} (proved in \cite{K1, K2, DML, Bed, Kuroda}), and one of the results by Drensky and Makar-Limanov \cite{DML}. 
Let the Weitzenb\"ock derivation $\delta$ act on $X_{2d},A_{2d},U_{2d},V_{2d}$
by the rule
\[
\delta(x_{2i})=x_{2i-1}, \delta(x_{2i-1})=0,\quad \delta(a_{2i})=a_{2i-1}, \delta(a_{2i-1})=0,
\]
\[
\delta(u_{2i})=u_{2i-1}, \delta(u_{2i-1})=0,\quad\delta(v_{2i})=v_{2i-1}, \delta(v_{2i-1})=0,
\]
for $i=1,\ldots,d$. Then the algebra of constants $K[U_{2d},V_{2d}]^{\delta}$ is generated by 
\[
u_1,u_3,\ldots,u_{2d-1},v_1,v_3,\ldots,v_{2d-1}
\]
and the determinants
\[
\alpha_{pq}=u_{2p-1}u_{2q}-u_{2p}u_{2q-1}=\begin{vmatrix}
u_{2p-1} & u_{2p} \\
u_{2q-1} &u_{2q}
\end{vmatrix},\quad 1\leq p<q\leq d,
\]
\[
\beta_{pq}=v_{2p-1}v_{2q}-v_{2p}v_{2q-1}=\begin{vmatrix}
v_{2p-1} & v_{2p} \\
v_{2q-1} &v_{2q}
\end{vmatrix},\quad1\leq p<q\leq d,
\]
\[
\gamma_{pq}=u_{2p-1}v_{2q}-u_{2p}v_{2q-1}=\begin{vmatrix}
u_{2p-1} & u_{2p} \\
v_{2q-1} &v_{2q}
\end{vmatrix}, \quad p,q=1,\ldots,d,
\]
with the following defining relations
\begin{equation} \label{S1}
u_{2i-1}\alpha_{jk}-u_{2j-1}\alpha_{ik}+u_{2k-1}\alpha_{ij}=0,\quad 1\leq i<j<k\leq d,
\end{equation}
\begin{equation} \label{S2}
u_{2i-1}\gamma_{jk}-u_{2j-1}\gamma_{ik}+v_{2k-1}\alpha_{ij}=0,\, 1\leq i<j\leq d,\, 1\leq k\leq d,
\end{equation}
\begin{equation} \label{S3}
u_{2i-1}\beta_{jk}-v_{2j-1}\gamma_{ik}+v_{2k-1}\gamma_{ij}=0,\, 1\leq i\leq d,\, 1\leq j<k\leq d,
\end{equation}
\begin{equation} \label{S4}
v_{2i-1}\beta_{jk}-v_{2j-1}\beta_{ik}+v_{2k-1}\beta_{ij}=0,\quad 1\leq i<j<k\leq d,
\end{equation}
\begin{equation}  \label{R1}
\alpha_{ij}\alpha_{kl}-\alpha_{ik}\alpha_{jl}+\alpha_{il}\alpha_{jk}=0,\, 1\leq i<j<k<l\leq d,
\end{equation}
\begin{equation}  \label{R2}
\alpha_{ij}\gamma_{kl}-\alpha_{ik}\gamma_{jl}+\gamma_{il}\alpha_{jk}=0,\, 1\leq i<j<k\leq d,\, 1\leq l\leq d,
\end{equation}
\begin{equation}  \label{R3}
\alpha_{ij}\beta_{kl}-\gamma_{ik}\gamma_{jl}+\gamma_{il}\gamma_{jk}=0,\,1\leq i<j\leq d,\,1\leq k<l\leq d,
\end{equation}
\begin{equation}  \label{R4}
\gamma_{ij}\beta_{kl}-\gamma_{ik}\beta_{jl}+\gamma_{il}\beta_{jk}=0,\,1\leq i\leq d,\,1\leq j<k<l\leq d,
\end{equation}
\begin{equation}  \label{R5}
\beta_{ij}\beta_{kl}-\beta_{ik}\beta_{jl}+\beta_{il}\beta_{jk}=0,\quad 1\leq i<j<k<l\leq d.
\end{equation}
The vector space $K[U_{2d},V_{2d}]^{\delta}$
has a {\it canonical} linear basis consisting of the elements of the form
\begin{equation}  \label{B1}
v_{2i_1-1}\cdots v_{2i_m-1}\beta_{p_1q_1}\cdots \beta_{p_rq_r}
\gamma_{p'_1q'_1}\cdots \gamma_{p'_sq'_s}\alpha_{p''_1q''_1}\cdots \alpha_{p''_tq''_t}u_{2j_1-1}\cdots u_{2j_n-1}
\end{equation}
such that among the generators $\beta_{pq}$, $\gamma_{p'q'}$, and $\alpha_{p''q''}$ there is no \textit{intersection},
and no one \textit{covers} $v_{2i_k-1}$ or $u_{2j_l-1}$.

Note that each $\beta_{pq}$, $\gamma_{p'q'}$, and $\alpha_{p''q''}$ is identified with the open interval
$(p+d,q+d)$, $(p',q'+d)$, and $(p'',q'')$, respectively, on the real line. The generators \textit{intersect} each other if the corresponding
open intervals have a nonempty intersection and are not contained in each other. On the other hand, the generators
$v_{2i-1}$ and $u_{2j-1}$ are identified with the points $i+d$ and $j$, respectively.
We say also that a generator among $\beta_{pq}$, $\gamma_{p'q'}$, or $\alpha_{p''q''}$ \textit{covers}
$v_{2i-1}$ or $u_{2j-1}$ if the corresponding open interval covers the corresponding point.

The pairs of indices are ordered in the following way: $p_1 \leq \cdots \leq p_r$ and if $p_{\xi}=p_{{\xi}+1}$, then $q_{\xi}\leq q_{{\xi}+1}$;
$p'_1 \leq \cdots \leq p'_s$ and if $p'_{\mu}=p'_{{\mu}+1}$, then $q'_{\mu}\leq q'_{{\mu}+1}$;
$p''_1 \leq \cdots \leq p''_t$ and if $p''_{\sigma}=p''_{{\sigma}+1}$, then $q''_{\sigma}\leq q''_{{\sigma}+1}$;
and $i_1 \leq \cdots \leq i_m$, $j_1 \leq \cdots \leq j_n$.

In order to detect the constants in the commutator ideal $F_{2d}'$ it is sufficient to work in
the $K[U_{2d},V_{2d}]^{\delta}$-submodule $C_{2d}^{\delta}$ of $M_{2d}^{\delta}$, which is generated by 
\[
a_1,a_3,\ldots,a_{2d-1}
\]
and the determinants
\[
w_{pq}=a_{2p-1}v_{2q}-a_{2q}v_{2p-1}=\begin{vmatrix}
a_{2p-1} & a_{2q} \\
v_{2p-1} &v_{2q}
\end{vmatrix}, \quad p,q=1,\ldots,d.
\]
and spanned as a vector space on the elements of the form
\begin{equation}  \label{sp1}
a_{2i_0-1}v_{2i_1-1}\cdots v_{2i_m-1}\beta_{p_1q_1}\cdots \beta_{p_rq_r}
\gamma_{p'_1q'_1}\cdots \gamma_{p'_sq'_s}\alpha_{p''_1q''_1}\cdots \alpha_{p''_tq''_t}u_{2j_1-1}\cdots u_{2j_n-1}
\end{equation}
\begin{equation}  \label{sp2}
w_{p_0q_0}v_{2i_1-1}\cdots v_{2i_m-1}\beta_{p_1q_1}\cdots \beta_{p_rq_r}
\gamma_{p'_1q'_1}\cdots \gamma_{p'_sq'_s}\alpha_{p''_1q''_1}\cdots \alpha_{p''_tq''_t}u_{2j_1-1}\cdots u_{2j_n-1}
\end{equation}
for each $i_0,p_0,q_0=1\ldots,d$.

We also have the following relations in the algebra $C_{2d}^{\delta}$ as a consequence of \cite{DML}.
\begin{equation} \label{S}
a_{2i-1}\beta_{jk}-w_{ik}v_{2j-1}+w_{ij}v_{2j-1}=0,\quad 1\leq i\leq d, \,1\leq j<k\leq d,
\end{equation}
\begin{equation}  \label{R}
w_{ij}\beta_{kl}-w_{ik}\beta_{jl}+w_{il}\beta_{jk}=0,\quad 1\leq i\leq d, \,1\leq j<k<l\leq d.
\end{equation}

We denote by $L$ the $K[U_{2d},V_{2d}]^{\delta}$-submodule of $C_{2d}^{\delta}$
generated by the following elements
\begin{equation} \label{g1}
w_{ii},\quad \quad 1\leq i\leq d,
\end{equation}
\begin{equation} \label{g2}
w_{ij}+w_{ji},\quad 1\leq i<j\leq d,
\end{equation}
\begin{equation}\label{g3}
a_{2i-1}v_{2j-1}-a_{2j-1}v_{2i-1},\quad 1\leq i<j\leq d,
\end{equation}
\begin{equation}\label{g4}
a_{2i-1}\beta_{pq}-w_{pq}v_{2i-1},\, 1\leq i\leq d, \, 1\leq p<q\leq d,
\end{equation}
\begin{equation}\label{g5}
w_{ij}\beta_{pq}-w_{pq}\beta_{ij}, \, 1\leq i<j\leq d,  \, 1\leq p<q\leq d,
\end{equation}
\begin{equation} \label{g6}
a_{2i-1}\beta_{jk}-a_{2j-1}\beta_{ik}+a_{2k-1}\beta_{ij},\quad 1\leq i<j<k\leq d,
\end{equation}
\begin{equation}  \label{g7}
w_{kl}\gamma_{ij}-w_{jl}\gamma_{ik}+w_{jk}\gamma_{il},\,1\leq i\leq d,\,1\leq j<k<l\leq d.
\end{equation}
\begin{equation} \label{g8}
w_{jk}u_{2i-1}-a_{2j-1}\gamma_{ik}+a_{2k-1}\gamma_{ij},\, 1\leq i\leq d,\, 1\leq j<k\leq d.
\end{equation}
One can observe that the generating elements (\ref{g1})-(\ref{g8}) of $L$ are
the images of some elements in the commutator ideal $F_{2d}'$ of the free associative algebra $F_{2d}$.

\begin{theorem}
The $K[U_{2d},V_{2d}]^{\delta}$-submodule $L$ of $C_{2d}^{\delta}$ consists of all commutator elements in $C_{2d}^{\delta}$;
i.e., the images in $(F_{2d}')^{\delta}\subset F_{2d}'$.
\end{theorem}


As a consequence of the previous result it is possible to get the full list of generators of $(F'_{2d})^\delta$ as a $K[U_{2d},V_{2d}]$-module.

\begin{corollary}\label{generators in the ideal}
The $K[U_{2d},V_{2d}]^{\delta}$-module $(F_{2d}')^{\delta}$ is generated by the following polynomials:
\[
g_1(i)=[x_{2i-1},x_{2i}],\quad \quad 1\leq i\leq d,
\]
\[
g_2(i,j)=[x_{2i-1},x_{2j-1}],\quad 1\leq i<j\leq d,
\]
\[
g_3(i,j)=[x_{2i-1},x_{2j}]+[x_{2j-1},x_{2i}],\quad 1\leq i<j\leq d,
\]
\[
g_4(i,p,q)=[x_{2i-1},x_{2p-1},x_{2q}]-[x_{2i-1},x_{2p},x_{2q-1}],\, 1\leq i\leq d, \, 1\leq p<q\leq d,
\]
\[
g_5(i,j,k)=[x_{2i-1},x_{2j-1},x_{2k}]-[x_{2i-1},x_{2k-1},x_{2j}]+[x_{2j-1},x_{2k-1},x_{2i}],
\]
\[
1\leq i<j<k\leq d,
\]
\[
g_6(i,j,p,q)=[x_{2i-1},x_{2p-1},x_{2j},x_{2q}]+[x_{2i},x_{2p},x_{2j-1},x_{2q-1}]
\]
\[
-[x_{2i-1},x_{2p},x_{2j},x_{2q-1}]-[x_{2i},x_{2p-1},x_{2j-1},x_{2q}],
\]
\[
1\leq i<j\leq d,\,1\leq p<q\leq d,
\]
\[
g_7(i,j,k,l)=x_{2i}[x_{2j-1},x_{2k-1},x_{2l}]+x_{2i-1}[x_{2j},x_{2k},x_{2l-1}]
\]
\[
-x_{2i}[x_{2j-1},x_{2k},x_{2l-1}]-x_{2i-1}[x_{2j},x_{2k-1},x_{2l}]
\]
\[
1\leq i\leq d,\,1\leq j<k<l\leq d,
\]
\[
g_8(i,j,k)=x_{2i}[x_{2j-1},x_{2k-1}]-x_{2i-1}[x_{2j},x_{2k-1}],
\]
\[
1\leq i\leq d,\, 1\leq j<k\leq d.
\]
\end{corollary}

We have to add for the generating set of the whole algebra
$(F_{2d})^\delta$ the constants $x_{2i-1}$ and $x_{2i-1}x_{2j}-x_{2i}x_{2j-1}$, $1\leq i<j\leq d$, 
which are needed for the generation
of the factor algebra of $(F_{2d})^\delta$ modulo the commutator ideal of $F_{2d}$.
These generators are the ones lifted from the algebra $K[X_{2d}]^{\delta}$ of constants
of the polynomial algebra to the algebra $(F_{2d})^\delta$  by the fact stated in Corollary 4.3 of the paper \cite{DG} by Drensky and Gupta.

The following result gives an infinite generating set of the subalgebra of constants $(F_{2d})^\delta$ of the free metabelian associative algebra $F_{2d}$ as an algebra.

\begin{corollary}
The algebra $(F_{2d})^\delta$ of the constants is generated by
\[
x_1,x_3,\ldots,x_{2d-1},
\]
\[
x_{2i-1}x_{2j}-x_{2i}x_{2j-1},
\]
\[
g_1f_1,\ldots,g_8f_8,
\]
where $1\leq i<j\leq d$, $f_1,\ldots,f_8\in K[U_{2d},V_{2d}]^{\delta}$, and $g_i$s are as in Corollary \ref{generators in the ideal}.
\end{corollary}

\section{The variety $\mathcal{G}$ generated by the Grassmann algebra}
In this section we shall show results concerning Nowicki's conjecture of the relatively free algebra of the infinite dimensional Grassmann algebra.

The variety $\mathcal{G}$ consists of all associative unitary algebras satisfying
the polynomial identity $[z_1,z_2,z_3]\equiv0$.
We shall set $F_{2l}:=F_{2l}(\mathcal{G})$ for $0\leq l\leq d$.
As noted in the Introduction, $\mathcal{G}$ is the variety generated by
the infinite dimensional Grassmann algebra $G$ and $F_{2d}$ coincides
with the $2d$-generated relatively free algebra in the variety $\mathcal{G}$. 
The identities of $G$ and related topics have been studied by several authors.
See for example the paper \cite{KR} by Krakovski and Regev about the ideal of polynomial identities of $G$.

Let $X=\{x_1,x_2,\ldots,x_{d}\}$, $Y=\{y_1,y_2,\ldots,y_{d}\}$ be two disjoint sets of variables. Of course the relatively free algebra of $\mathfrak{G}$ of rank $2d$ in the variables from $U:=X\cup Y$ is isomorphic to $F_{2d}$ and we consider the following order inside $U$: \[x_1<y_1<x_2<\cdots<x_d<y_d.\] 

Moreover we say a polynomial $f$  is \textit{homogeneous in the set of variables} $S=\{u_{1},\ldots,u_s\}$ if for each monomial $m$ appearing in $f$ we have $\sum_{i=1}^s\deg_{u_i}m$ is the same.

It is well known (see for example Theorem 5.1.2 of \cite{D}) $F_{2d}$ has a basis consisting of all
\[
x_1^{a_1}y_1^{b_1}\cdots x_{d}^{a_{2d}}y_{d}^{b_{2d}}[u_{i_1},u_{i_2}]\cdots[u_{i_{2c-1}},u_{i_{2c}}],\]\[a_i,b_j\geq0,\ ,u_{i_l}\in U,\ u_{i_1}<u_{i_2}<\cdots<u_{i_{2c}},\ c\geq0.\]

We recall the identity $[z_1,z_2,z_3]\equiv0$ implies the identity
\begin{equation}\label{proof Grass}
[z_1,z_2][z_3,z_4]\equiv-[z_1,z_3][z_2,z_4]
\end{equation}
in $F_{2d}$.

Consider the following Weitzenb\"ock derivation $\delta$ of $F_{2d}$ acting on 
 $U$ such that
\[
\delta(y_{i})=x_{i}, \ \delta(x_i)=0,\quad 1\leq i\leq d.
\]

Let $\alpha=(\alpha_1,\cdots,\alpha_{d-1})\in K^{d-1}$ and consider the algebra endomomorphism $\phi_{\alpha}$ of $F_{2d}$ such that \[\phi_\alpha(x_i)=x_i,\ \ \phi_\alpha(y_i)=y_i,\ \ i=1,\ldots,d-1,\]
\[\phi_\alpha(x_d)=\sum_{i=1}^{d-1}\alpha_ix_i,\ \ \phi_\alpha(y_d)=\sum_{i=1}^{d-1}\alpha_iy_i.\] Notice that $\phi_\alpha$ commutes with $\delta$. Hence if $f\in F_{2d}^\delta$, then $\phi_\alpha(f)\in F_{2d-2}^\delta$ too. Moreover, if $f\in F_{2d}$ is homogeneous with respect to the set $\{x_d,y_d\}$ and $\phi_\alpha(f)=0$ for some non-zero $\alpha\in K^{d-1}$, then $f$ is in the left ideal generated by \[\omega_\alpha:=\left(\sum_{i=1}^{d-1}\alpha_ix_i\right)y_d-\left(\sum_{i=1}^{d-1}\alpha_iy_i\right)x_d,\ \ \ [\omega_\alpha,u],\ \ u\in U,\]\[\mu_\alpha:=\left[x_d,\sum_{i=1}^{d-1}\alpha_ix_i\right],\ \ \nu_\alpha:=\left[y_d,\sum_{i=1}^{d-1}\alpha_iy_i\right].\]

We define the following objects inside
$F_{2d}$:
\[
v_{ij}:=x_iy_j-y_ix_j, \quad 1\leq i,j\leq d,
\]
\[w_{ijk}:=y_i[x_j,x_k]-x_i[y_j,x_k], \quad 1\leq i,j,k\leq d,\]
\[z_{ijkl}:=y_i[x_j,v_{k,l}]-x_i[y_j,v_{k,l}], \quad 1\leq i,j,k,l\leq d,\]

Notice that the $v_{ij}$'s, the $w_{ijk}$'s and the $z_{ijkl}$'s are constants of $F_{2d}$ and starting from these objects we shall construct three subsets of elements of $F_{2d}$. We set \[V:=\{v_{ij}|1\leq i,j\leq d\},\]\[W_0:=\{w_{ijk}|1\leq i,j\leq k\leq d\}\] and \[Z_0:=\{z_{ijkl}|1\leq i\leq j\leq k\leq l\leq d\}.\] Suppose $s>0$ and we set \[W_s:=\{y[x,w]-x[y,w]|w\in W_{s-1},\ x\in X,\ y\in Y\},\] \[Z_s:=\{y[x,z]-x[y,z]|z\in Z_{s-1}\ x\in X,\ y\in Y\}.\]
As above, both the $W_s$'s and the $Z_s$'s are subsets of constants. Moreover it can be easily seen for $s\geq d$ both $V_s$ and $W_s$ are 0. We shall denote by $\mathcal{C}$ the algebra generated by the non-zero $X,V,W_l,Z_l$. Here we have the main result. We recall the proof of the following follows the main steps of the proof of the classical Nowicki's conjecture presented in Section 3. We also point out this proof cannot be adapted to algebras having a non $T$-prime ideal of their polynomial identities.

\begin{theorem}
The algebra of constants $F^\delta_{2d}$ is finitely generated as an algebra, and its generators are:
\begin{equation}\label{grass1} X;\end{equation}\begin{equation}\label{grass2} V,\end{equation}
\begin{equation}\label{grass3} W_{l},\ l\leq d-1; \end{equation}
\begin{equation}\label{grass4} Z_{l},\ l\leq d-1. \end{equation}
\end{theorem} 

\begin{proof}
It is easy to see that the elements from (\ref{grass1})-(\ref{grass4}) are constants of ${F}_{2d}$.

Let $f=f(X',Y',x_d,y_d)\in F_{2d}^\delta$, then we may assume $f$ being homogeneous in $U$ and in $\{x_d,y_d\}$. We shall prove the theorem by induction on $d$ and on the total degree with respect to the set $\{x_d,y_d\}$. Suppose $d=1$, then \[f=\sum_{k=0}^n\alpha_k x^ky^{n-k}+\sum_{j=1}^{n-1}\beta_j x^{j-1}y^{n-j-1}[x,y].\] This means \[0=\delta(f)=\sum_{k=0}^{n-1}\alpha_k(n-k)x^{k+1}y^{n-k-1}\]\[+[x,y]\sum_{k=1}^{n-2}(n-k-1)\left(\alpha_k\frac{n-k}{2}x^{k-1}y^{n-k-1}+\beta_k x^ky^{n-k-2}\right).\] The previous relation gives us $\alpha_k=0$ for $k\leq n-1$, then $\beta_j=0$ for $1\leq j\leq n-2$. Thus the only possibility is $f=\alpha_nx^n+\beta_{n-1}x^{n-2}[x,y]$ and we are done.

Assume now $d>1$ and the result true for $F_{2d-2}$. We set $U'=X'\cup Y'$. If $v\in F_{2d-2}^\delta$ is homogeneous in $U$, then we shall study only the case $\deg_{x_d,y_d}(f)>0$.

Remark we can write $f$ in the following way:\[f=a_py_d^p+\sum_{u_0}a'_{p,u_0}[u_0,y_d]y_d^{p-1}+a_{p-1}x_dy_d^{p-1}+\sum_{u^1_1}a'_{p-1,u^1_1}[u^1_1,x_d]y_d^{p-1}\]\[+\sum_{u^2_1}a^{''}_{p-1,u^2_1}[u^2_1,y_d]x_dy_d^{p-2}+a^{'''}_{p-1}[x_d,y_d]y_d^{p-2}+\cdots+a_0x_d^p+\sum_{u_p}a'_{0}[u_p,x_d]x_d^{p-1},\] where the $u^i_j$'s are in $U$ whereas the $a_j$'s, the $a'_{r,u^i_j}$'s, the $a^{''}_{s,u^i_j}$ and the $a^{'''}_{t,u^i_j}$'s all belong to $F_{2d-2}$. Deriving $f$ we get $\delta(a_p)=0$, hence:

\begin{equation}\label{obs1}
a_p\in F_{2d}^\delta. 
\end{equation}
We also have $\frac{p(p-1)}{2}a_p+\delta(a^{'''}_{p-1})=0$, then by (\ref{obs1}) we get $\delta(a^{'''}_{p-1})=0$. This means:
\begin{equation}\label{obs2}
a^{'''}_{p-1}\in F_{2d}^\delta. 
\end{equation}
Suppose now $u_0=y_s$, $s\leq d-1$, then $a'_{p,u_0}=0$. So we can assume $u_0=x_s$, $s\leq d-1$. In this case we get $\delta(a'_{p,x_s})=0$, i.e.:
\begin{equation}\label{obs3}
a'_{p,x_s}\in F_{2d}^\delta.\end{equation}
By (\ref{obs1}), (\ref{obs2}) and (\ref{obs3}), using induction, we get \[a_p=\sum b_1(X',V',W',Z'),\]\[a^{'''}_{p-1}=\sum b_2(X',V',W',Z'),\]\[a'_{p,x_i}=\sum b_{3,i}(X',V',W',Z').\]

If $\phi_\alpha(f)=0$ for all $\alpha=(\alpha_1,\ldots,\alpha_{d-1})$, then we get $d=2$ and $f$ is in the left ideal generated by $v_{12}$ and $[v_{12},u]$, $u\in U$. Hence $f=gv_{12}+\sum_{u\in U}g_u[v_{12},u]$ and we apply inductive arguments to $g$ and $g_u$'s and we are done. Now we consider the case $\phi_\alpha(f)\neq0$ for some non-zero $\alpha\in K^{n-1}$. The next two remarks are crucial. First we have \[\deg_X(f)=\deg_{X'}(\phi_\alpha(f))\geq \deg_{Y'}(\phi_\alpha(f))=\deg_{Y}(f).\] Hence \[\deg_X(f)=\deg_{X'}(a_p)\geq\deg_Y(f)=p+\deg_{Y'}(a_p)\] so \[a_p=\sum b'_1(X',V',W',Z')x_{k_1}\cdots x_{k_p}.\] We can argue analogously and obtain \[a^{'''}_{p-1}=\sum b'_2(X',V',W',Z')x_{l_1}\cdots x_{l_p},\]\[a'_{p,x_i}=\sum b_{3,i}(X',V',W',Z')x_{m_1}\cdots x_{m_p}.\]

Now we rewrite $a_py_d^p$, $a^{'''}_{p-1}[x_d,y_d]y_d^{p-2}$ and $a'_{p,x_i}[x_i,y_d]y_d^{p-1}$ and we get \[f=c+gx_d+\sum_{i=1}^{d-1}m_{p,i}y_d[x_i,x_d],\] where $c\in\mathcal{C}$. We can rewrite the last summand as $g_1x_d+\sum_{i=1}^{d-1}h_i[x_i,x_d]$, where $\deg_{x_d}h_i=0$ and we get finally \[f=c+g'x_d+\sum_{i=1}^{d-1}h_i[x_i,x_d].\] Hence $g'$ and the $h_i$'s are constants too because $x_d$ and $[x_i,x_d]$ are. Now we are allowed to apply induction on $g'$ and the $h_i$'s because their degrees with respect to the set $\{x_d,y_d\}$ are strictly smaller than the one of $f$ and we conclude the proof.
\end{proof}

\section{The Nowicki conjecture for the free metabelian Poisson algebras}
This section presents a new result and is devoted to the study of Nowicki's conjecture in a nonassociative setting. In particular, we will handle the case of a free metabelian Poisson algebra and we shall give an explicit set of basis elements of the vector space of its constants.

Let $K$ be a field of characteristic zero. Let $(P,\cdot)$ and $(P,[\ ,\ ])$ be a commutative associative algebra and a Lie algebra structure on the $K$-vector space $P$, respectively, over $K$ such that the identity
\[
[a,b\cdot c]=[a,b]\cdot c+b\cdot [a,c]=[a,b]\cdot c+[a,c]\cdot b
\]
holds in $P$. Then $(P,\cdot,[\ ,\ ])$ is called a \textit{Poisson algebra}.

We will consider throughout the section only left normed commutators, i.e., $[a,b,c]:=[[a,b],c]$.
If $a\cdot b=[a,b]=0$ holds in $P$, then the Poisson algebra $P$ is called \textit{abelian} and any extension of $P$ by an abelian Poisson algebra is called \textit{metabelian} (see e.g. \cite{AM}). 

Let us consider the free metabelian Poisson algebra $P_{2n}$ over $K$ of rank $2n$ generated by the set
\[
Z_{2n}=\{z_1,\ldots,z_{2n}\}=\{x_1,\ldots,x_n,y_1,\ldots,y_n\}=X_n\cup Y_n.
\]
It is well known (see e.g. \cite{Zh}) that the following identities hold in $P_{2n}$:
\begin{align}
0&=a\cdot b\cdot c\cdot d=[a\cdot b, c\cdot d]=[[a,b], c\cdot d]\nonumber\\
&=[a,b]\cdot c \cdot d=[a,b]\cdot[c,d]=[[a,b],[c,d]],\nonumber
\end{align}
which imply the following ones:
\[
[a,b,c]=[a,c,b]-[b,c,a]
\]
\[
[a,b,c]\cdot d=[d,b,c]\cdot a-[c,b,d]\cdot a
\]
\[
[a,b]+[b,a]=0 \ , \ [a,b,c,d]=[a,b,d,c]
\]
\[
[a,b,c]\cdot d=-[a,b,d]\cdot c \ , \ [a,b,c]\cdot c=0.
\]

We put an order "$<$" in the generating set so that
\[
z_1<\cdots<z_{2n}
\]
or
\[
x_1<\cdots<x_n<y_1<\cdots<y_n. 
\]
It is well known (see \cite{Zh}) ${\mathcal B}_{1}\cup {\mathcal B}_{2}\cup {\mathcal B}_{3}\cup {\mathcal B}_{4}$
forms a $K$-basis for the free metabelian Poisson algebra $P_{2n}$,
where
\begin{align}
&{\mathcal B}_{1}=Z_{2n}\cup\{[z_{i_1},\dots,z_{i_k}]\mid i_1>i_2\leq i_3\leq\cdots\leq i_k, 2\leq k \},\nonumber\\
&{\mathcal B}_{2}=\{z_i\cdot z_j\ ,\ z_i\cdot z_j\cdot z_k\mid 1\leq i\leq j\leq k\leq 2n\},\nonumber\\
&{\mathcal B}_{3}=\{[z_i, z_j]\cdot z_k\mid 1\leq j\leq i\leq 2n \ , \ 1\leq k\leq 2n\},\nonumber\\
&{\mathcal B}_{4}=\{[z_i, z_j, z_k]\cdot  z_l \mid 1\leq j<i\leq l\leq 2n\ ,\ 1\leq j\leq k<l\leq 2n\}.\nonumber
\end{align}

Now consider the linear nilpotent derivation $\delta:y_i\to x_i\to0$, $i=1,\ldots,n$. Clearly the spaces $K{\mathcal B}_{i}$, $i=1,2,3,4$,
are $\delta$-invariant, i.e., $(K{\mathcal B}_{i})^{\delta}\subset K{\mathcal B}_{i}$, then
\[
P_{2n}^{\delta}=(K{\mathcal B}_{1})^{\delta}\oplus (K{\mathcal B}_{2})^{\delta}\oplus (K{\mathcal B}_{3})^{\delta}\oplus (K{\mathcal B}_{4})^{\delta}.
\]
Note that $K{\mathcal B}_{1}=L_{2n}$ is the free metabelian Lie algebra; moreover, by \cite{DF1}, we know the generators of the algebra $(L_{2n}')^{\delta}$
as a $K[X_n,Y_n]^{\delta}$-module:

\begin{theorem}\label{principal}
The $K[X_n,Y_n]^{\delta}$-module $(L_{2n}')^{\delta}$ is generated by the following elements
\[
[x_i,y_i],\quad \quad 1\leq i\leq n,
\]
\[
[x_i,x_j],\quad 1\leq i<j\leq n,
\]
\[
[x_i,y_j]+[x_j,y_i],\quad 1\leq i<j\leq n,
\]
\[
[x_i,x_p,y_q]-[x_i,y_p,x_q],\, 1\leq i\leq n, \, 1\leq p<q\leq n,
\]
\[
[x_i,x_j,y_k]-[x_i,x_k,y_j]+[x_j,x_k,y_i],\quad 1\leq i<j<k\leq n,
\]
and
\[
[x_i,x_p,y_j,y_q]+[y_i,y_p,x_j,x_q]-[x_i,y_p,y_j,x_q]-[y_i,x_p,x_j,y_q],
\]
where $1\leq i<j\leq n$, $1\leq p<q\leq n$.
\end{theorem}

By Nowicki's conjecture we get $(K{\mathcal B}_{2})^{\delta}$ is spanned on $X_n$ and the elements of the form
\[
u_{i,j}=x_i\cdot y_j-x_j\cdot y_i \ , \ 1\leq i<j<k\leq n
\]
and
\[
u_{i,j}\cdot x_k \ , \ 1\leq i<j<k\leq n \ , \ 1\leq k\leq n
\]
with the following relations:
\[
x_i\cdot u_{j,k}-x_j\cdot u_{i,k}+x_k\cdot u_{i,j}=0 \ , \ 1\leq i<j<k\leq n.
\]

We want to highlight the fact that once again some "determinants" (the $u_{i,j}$ above) appear as generators of the algebra of constants as in the classical Nowicki's conjecture and its generalization. 

In the sequel we are going to show an explicit basis of $KB_3^\delta$ and $KB_4^\delta$ as vector spaces. Then the explicit basis of $P_{2n}^\delta$ will be an obvious corollary.

\begin{theorem}\label{KB_3}
The $K$-vector space $(K{\mathcal B}_3)^{\delta}$ is of a basis consisting of elements of the form
\[
[x_i, x_j]\cdot x_k \ , \  j<i,
\]
\[
[y_i, x_i]\cdot x_j \ , \  i,j,
\]
\[
[x_i, x_j]\cdot  y_k+[y_j, x_i]\cdot  x_k \ , \  j<i<k,
\]
\[
[y_i, x_j]\cdot  x_k+[y_j, x_i]\cdot  x_k \ , \  j<i<k.
\]
\end{theorem}

\begin{proof}
Let us set
\begin{align}
A_1=&\{[x_i, x_j]\cdot x_k\mid 1\leq j<i\leq n \ , \ 1\leq k\leq n \},\nonumber\\
A_2=&\{[y_i, y_j]\cdot y_k\mid 1\leq j<i\leq n \ , \ 1\leq k\leq n\},\nonumber\\
A_3=&\{[x_i, x_j]\cdot y_k\mid 1\leq j<i\leq n \ , \ 1\leq k\leq n\},\nonumber\\
A_4=&\{[y_i, y_j]\cdot x_k\mid 1\leq j<i\leq n \ , \ 1\leq k\leq n\},\nonumber\\
A_5=&\{[y_i, x_j]\cdot y_k\mid 1\leq i,j,k\leq n\},\nonumber\\
A_6=&\{[y_i, x_j]\cdot x_k\mid 1\leq i,j,k\leq n\}.\nonumber
\end{align}
Indeed $K{\mathcal B}_3=KA_1\oplus KA_2\oplus KA_3\oplus KA_4\oplus KA_5\oplus KA_6$,
while $(KA_1)^{\delta}=KA_1$,
\[
(KA_2)^{\delta}\subset KA_4\oplus KA_5 \ , \ (KA_3)^{\delta}\subset KA_1,
\]
\[
(KA_4)^{\delta}\subset KA_6 \ , \ (KA_5)^{\delta}\subset KA_3\oplus KA_6 \ , \ (KA_6)^{\delta}\subset KA_1.
\]
Thus
\[
(K{\mathcal B}_3)^{\delta}=KA_1\oplus(KA_2)^{\delta}\oplus(KA_3\oplus KA_6)^{\delta}\oplus(KA_4\oplus KA_5)^{\delta}.
\]
Now let
\[
p_2=\sum_{j<i}\alpha_{ijk}[y_i, y_j]\cdot y_k
\]
be an element in $(KA_2)^{\delta}$, then $\delta(p_2)=0$ and
\begin{align}
0=&\sum_{j<i}\alpha_{ijk}([x_i, y_j]\cdot y_k+[y_i,x_j]\cdot y_k+[y_i, y_j]\cdot x_k)\nonumber\\
=&\sum_{j<i}\alpha_{ijk}(-[y_j,x_i]\cdot y_k+[y_i,x_j]\cdot y_k)+\sum_{j<i}\alpha_{ijk}[y_i, y_j]\cdot x_k.\nonumber
\end{align}
Notice that $\sum_{j<i}\alpha_{ijk}[y_i, y_j]\cdot x_k$ belongs to $KA_4$, then
\[
0=\sum_{j<i}\alpha_{ijk}[y_i, y_j]\cdot x_k.
\]
We immediately get any $\alpha_{ijk}$ is 0 and $p_2$ turns out to be 0, too. This proves that there is no nonzero constant in $KA_2$.

Suppose now $p_3+p_6\in (KA_3\oplus KA_6)^{\delta}$ for some $p_3\in KA_3$ and $p_6\in KA_6$.
First observe that elements of the form $[y_i, x_i]\cdot x_k$ are constants in $KA_6$. Then we may assume
 \[
p_3=\sum_{j<i}\alpha_{ijk}[x_i, x_j]\cdot y_k \ \in KA_3 \ \ \ \text{\rm and} \ \ \ p_6=\sum_{i\neq j}\beta_{ijk}[y_i, x_j]\cdot x_k \ \in KA_6.
\]
Hence
\begin{align}
\delta(p_3+p_6)=0=&\sum_{j<i}\alpha_{ijk}[x_i, x_j]\cdot x_k+\sum_{j<i}\beta_{ijk}[x_i, x_j]\cdot x_k+\sum_{i<j}\beta_{ijk}[x_i, x_j]\cdot x_k\nonumber\\
=&\sum_{j<i}(\alpha_{ijk}+\beta_{ijk}-\beta_{jik})[x_i, x_j]\cdot x_k\nonumber
\end{align}
and $\beta_{jik}=\alpha_{ijk}+\beta_{ijk}$ for $j<i$. This gives us $(KA_3\oplus KA_6)^{\delta}$
is spanned on the elements of the form
\[
[x_i, x_j]\cdot y_k+[y_j, x_i]\cdot x_k \ \ \ \text{\rm and} \ \ \ [y_i, x_j]\cdot x_k+[y_j, x_i]\cdot x_k
\]
for $j<i$.

We simply need to show $(KA_4\oplus KA_5)^{\delta}=\{0\}$. Let $p_4+p_5\in (KA_4\oplus KA_5)^{\delta}$ such that
 \[
p_4=\sum_{j<i}\alpha_{ijk}[y_i, y_j]\cdot x_k
\]
and
\begin{align}
p_5=&\sum_{}\beta_{ijk}[y_i, x_j]\cdot y_k=\sum_{j<i}(\beta_{ijk}[y_i, x_j]\cdot y_k+\beta_{jik}[y_j, x_i]\cdot y_k)+\sum_{}\beta_{iik}[y_i, x_i]\cdot y_k.\nonumber
\end{align}
We have
\[
\delta(p_4)=\sum_{j<i}\alpha_{ijk}([x_i, y_j]\cdot x_k+[y_i,x_j]\cdot x_k),
\]
\begin{align}
\delta(p_5)=&\sum_{j<i}\left(\beta_{ijk}([x_i, x_j]\cdot y_k+[y_i, x_j]\cdot x_k)-\beta_{jik}([x_j, x_i]\cdot y_k+[y_j, x_i]\cdot x_k)\right)\nonumber\\
&+\sum_{}\beta_{iik}[y_i, x_i]\cdot x_k.\nonumber
\end{align}
Then $\delta(p_4+p_5)=0$ gives that
\[
0=\sum_{j<i}(\alpha_{ijk}+\beta_{ijk})[y_i,x_j]\cdot x_k+(-\alpha_{ijk}+\beta_{jik})[y_j,x_i]\cdot x_k \ \in KA_6
\]
\[
0=\sum_{}\beta_{iik}[y_i, x_i]\cdot x_k \ \in KA_6 \ , \  0=\sum_{j<i}(\beta_{ijk}-\beta_{jik})[x_i, x_j]\cdot y_k \ \in KA_3
\]
that is
\[
\alpha_{ijk}+\beta_{ijk}=0 \ , \ -\alpha_{ijk}+\beta_{jik}=0 \ , \ \beta_{iik}=0 \ , \ \beta_{ijk}-\beta_{jik}=0;
\]
hence
\[
\alpha_{ijk}=\beta_{ijk}=\beta_{jik}=\beta_{iik}=0
\]
for $j<i$. Therefore, $p_4=p_5=0$. Now the proof follows by direct computations showing that the constants found are linearly independent.
\end{proof}

\begin{theorem}\label{KB_4}
The $K$-vector space $(K{\mathcal B}_4)^{\delta}$ is of a basis consisting of elements of the form
\[
[x_i, x_j, x_k]\cdot  x_l \ , \  j<i\leq l\ ,\ j\leq k<l,
\]
\[
[x_i, x_j, x_j]\cdot  y_j \ , \ [x_i, x_j, x_i]\cdot  y_i \ , \ [x_i, x_j, x_i]\cdot  y_j+[x_i, x_j, x_j]\cdot  y_i \ , \ j<i,
\]
\[
[x_i, x_j, x_k]\cdot  y_k \ , \ [x_k, x_j, x_i]\cdot  y_i \ , \ [x_i, x_j, x_i]\cdot  y_k+[x_i, x_j, x_k]\cdot  y_i\ , \ j<i<k,
\]
\[
[x_i, x_j, x_i]\cdot  y_k+[x_k, x_i, x_i]\cdot  y_j\ , \ [x_i, x_j, x_k]\cdot  y_j-[x_k, x_j, x_i]\cdot  y_j\ , \  j<i<k,
\]
\[
[x_i, x_j, x_k]\cdot  y_j+[x_i, x_j, x_j]\cdot  y_k\ , \  [x_i, x_j, x_k]\cdot  y_j+[x_k, x_j, x_j]\cdot  y_i\ , \ j<i<k,
\]
\[
[x_k, x_j, x_k]\cdot  y_i+[x_k, x_j, x_i]\cdot  y_k \ , \ [x_k, x_j, x_k]\cdot  y_i-[x_k, x_i, x_k]\cdot  y_j\ , \  j<i<k,
\]
\[
[x_k, x_j,x_l]\cdot y_i+[x_i, x_j,x_k]\cdot y_l-[x_k, x_i,x_l]\cdot y_j  \ , \  j<i<k<l,
\]
\[
[x_k, x_j,x_i]\cdot y_l-[x_i, x_j,x_k]\cdot y_l+[x_k, x_i,x_l]\cdot y_j  \ , \  j<i<k<l,
\]
\[
[x_l, x_i,x_k]\cdot y_j+[x_i, x_j,x_k]\cdot y_l-[x_k, x_i,x_l]\cdot y_j  \ , \  j<i<k<l,
\]
\[
[x_i, x_j,x_l]\cdot y_k+[x_i, x_j,x_k]\cdot y_l \ , \ [x_l, x_j,x_i]\cdot y_k+[x_k, x_i,x_l]\cdot y_j \ , \ j<i<k<l,
\]
\[
[x_l, x_j,x_k]\cdot y_i-[x_k, x_i,x_l]\cdot y_j \ , \ j<i<k<l,
\]
\[
[y_i, x_i, x_i]\cdot  y_i \ , \ [y_j, x_j, x_i]\cdot  y_i \ , \ [y_j, x_i, x_i]\cdot  y_i+[y_i, x_j, x_i]\cdot  y_i \ , \ j<i,
\]
\[
[y_j, x_j, x_i]\cdot  y_j+[y_j, x_j, x_j]\cdot  y_i \ , \ [y_j, x_i, x_i]\cdot  y_j-2[x_i, x_j, y_j]\cdot  y_i+[y_i, x_j, x_j]\cdot  y_i \ , \ j<i,
\]
\[
[y_i, x_j, x_k]\cdot  y_i+[x_i, x_j, y_i]\cdot  y_k+[x_k, x_i, y_j]\cdot  y_i +[y_j, x_i, x_i]\cdot  y_k \ , \ j<i<k,
\]
\[
[y_j, x_i, x_k]\cdot  y_i-[x_i, x_j, y_i]\cdot  y_k-[x_k, x_i, y_j]\cdot  y_i +[y_i, x_j, x_i]\cdot  y_k \ , \ j<i<k,
\]
\[
[x_k, x_j, y_i]\cdot  y_k+[x_k, x_i, y_j]\cdot  y_k -[y_j, x_k, x_k]\cdot  y_i -[y_k, x_j, x_i]\cdot  y_k \ , \ j<i<k,
\]
\[
[y_i, x_j, x_k]\cdot  y_k+[y_j, x_i, x_k]\cdot  y_k \ , \ [y_j, x_j, x_k]\cdot  y_i+[y_j, x_j, x_i]\cdot  y_k \ , \ j<i<k,
\]
\[
[x_i, x_j, y_j]\cdot  y_k+[x_k, x_j, y_j]\cdot  y_i -[y_i, x_j, x_j]\cdot  y_k-[y_j, x_i, x_k]\cdot  y_j \ , \ j<i<k,
\]
\[
[x_l, x_k, y_j]\cdot  y_i +[x_i, x_j, y_k]\cdot  y_l+[y_j, x_i, x_k]\cdot  y_l+[y_i, x_j, x_l]\cdot  y_k \ , \ j<i<k<l,
\]
\[
[y_j, x_k, x_l]\cdot  y_i+[y_k, x_j, x_i]\cdot  y_l-[x_l, x_i, y_j]\cdot  y_k-[x_k, x_j, y_i]\cdot  y_l  \ , \ j<i<k<l,
\]
\[
[y_j, x_i, x_l]\cdot  y_k+[y_i, x_j, x_k]\cdot  y_l +[y_j, x_i, x_k]\cdot  y_l+[y_i, x_j, x_l]\cdot  y_k \ , \ j<i<k<l,
\]
\[
[x_l, x_j, y_i]\cdot  y_k+[x_k, x_i, y_j]\cdot  y_l -[y_j, x_i, x_k]\cdot  y_l-[y_i, x_j, x_l]\cdot  y_k-[x_l, x_i, y_j]\cdot  y_k-[x_k, x_j, y_i]\cdot  y_l.
\]
\end{theorem}

\begin{proof}
It is easy to see
\[
K{\mathcal B}_4=KB_1\oplus KB_2\oplus KB_3\oplus KB_4\oplus KB_5\oplus KB_6,
\]
where
\begin{align}
B_1=&\{[x_i, x_j,x_k]\cdot x_l\mid 1\leq j<i\leq l\leq n \ , \ 1\leq j\leq k<l \leq n \},\nonumber\\
B_2=&\{[y_i, y_j,y_k]\cdot y_l\mid 1\leq j<i\leq l\leq n \ , \ 1\leq j\leq k<l \leq n \},\nonumber\\
B_3=&\{[x_i, x_j,x_k]\cdot y_l\mid 1\leq j<i \leq n \ , \ 1\leq j\leq k \leq n\},\nonumber\\
B_4=&\{[x_i, x_j,y_k]\cdot y_l\mid 1\leq j<i\leq n \ , \ 1\leq k<l \leq n\},\nonumber\\
B_5=&\{[y_i, x_j,x_k]\cdot y_l\mid 1\leq i\leq l\leq n \ , \ 1\leq j\leq k\leq n\},\nonumber\\
B_6=&\{[y_i, x_j,y_k]\cdot y_l\mid 1\leq i\leq l\leq n \ , \ 1\leq k<l \leq n\}.\nonumber
\end{align}
We have $(KB_1)^{\delta}=KB_1$,
\[
(KB_2)^{\delta}\subset KB_6 \ , \ (KB_3)^{\delta}\subset KB_1,
\]
\[
(KB_4)^{\delta}\subset KB_3 \ , \ (KB_5)^{\delta}\subset KB_3 \ , \ (KB_6)^{\delta}\subset KB_4\oplus KB_5
\]
and we get
\[
(K{\mathcal B}_4)^{\delta}=KB_1\oplus(KB_2)^{\delta}\oplus(KB_3)^{\delta}\oplus(KB_4\oplus KB_5)^{\delta}\oplus(KB_6)^{\delta}.
\]
We divide the proof in four cases: we study separately $(KB_2)^{\delta}$, $(KB_3)^{\delta}$, $(KB_4\oplus KB_5)^{\delta}$, and finally $(KB_6)^{\delta}$.

\

\noindent\underline{Case 1} Let $p_2\in (KB_2)^{\delta}$ be of the form
\begin{align}
p_2=&\sum_{j<i\leq l, j\leq k<l}\alpha_{ijkl}[y_i, y_j,y_k]\cdot y_l.\nonumber
\end{align}
Note that the condition $j<i\leq l \ , \ j\leq k<l$ can be divided into six parts:
\begin{align}
&j=k<i=l \ , \ j<k<i<l \ , \ j<i<k<l,\nonumber\\
&j<i=k<l \ , \ j<k<i=l \ , \ j=k<i<l.\nonumber
\end{align}
Hence
\begin{align}
p_2=&\sum_{j<i}\alpha_{ijji}[y_i, y_j,y_j]\cdot y_i+\sum_{j<k<i<l}\alpha_{ijkl}[y_i, y_j,y_k]\cdot y_l+\sum_{j<i<k<l}\alpha_{ijkl}[y_i, y_j,y_k]\cdot y_l\nonumber\\
&+\sum_{j<i<l}\alpha_{ijil}[y_i, y_j,y_i]\cdot y_l+\sum_{ j<k<i}\alpha_{ijki}[y_i, y_j,y_k]\cdot y_i+\sum_{j<i<l}\alpha_{ijjl}[y_i, y_j,y_j]\cdot y_l\nonumber\\
=&\sum_{j<i}\alpha_{ijji}[y_i, y_j,y_j]\cdot y_i+\sum_{j<i<k<l}(\alpha_{kjil}[y_k, y_j,y_i]\cdot y_l+\alpha_{ijkl}[y_i, y_j,y_k]\cdot y_l)\nonumber\\
&+\sum_{j<i<k}(\alpha_{ijik}[y_i, y_j,y_i]\cdot y_k+\alpha_{kjik}[y_k, y_j,y_i]\cdot y_k+\alpha_{ijjk}[y_i, y_j,y_j]\cdot y_k).\nonumber
\end{align}
Let us rewrite $p_2$ as
\[
p_2=p_{22}+p_{23}+p_{24},
\]
where
\[
p_{22}=\sum_{j<i}\alpha_{ijji}[y_i, y_j,y_j]\cdot y_i \ , \ p_{24}=\sum_{j<i<k<l}(\alpha_{kjil}[y_k, y_j,y_i]\cdot y_l+\alpha_{ijkl}[y_i, y_j,y_k]\cdot y_l),
\]
\[
p_{23}=\sum_{j<i<k}(\alpha_{ijik}[y_i, y_j,y_i]\cdot y_k+\alpha_{kjik}[y_k, y_j,y_i]\cdot y_k+\alpha_{ijjk}[y_i, y_j,y_j]\cdot y_k).
\]
Let $V_{2m}$, $m=2,3,4$, be the subspace of $KB_6$ spanned on monomials $[y_i, x_j,y_k]\cdot y_l$ with $|\{i,j,k,l\}|=m$.
Then $V_{22}\cap V_{23}=\{0\}$, $V_{22}\cap V_{24}=\{0\}$, $V_{23}\cap V_{24}=\{0\}$,
and $\delta(p_{22})\in V_{22}$, $\delta(p_{23})\in V_{23}$, $\delta(p_{24})\in V_{24}$.
Therefore $\delta(p_2)=0$ implies that $\delta(p_{22})=\delta(p_{23})=\delta(p_{24})=0$.

Remark that $\delta([y_i, y_j,y_j]\cdot y_i)$ can be written as a linear combination of basis elements as follows:
\begin{align}
\delta([y_i, y_j,y_j]\cdot y_i)=&[x_i, y_j,y_j]\cdot y_i+[y_i, x_j,y_j]\cdot y_i+[y_i, y_j,x_j]\cdot y_i+[y_i, y_j,y_j]\cdot x_i\nonumber\\
=&-[y_j,x_i, y_j]\cdot y_i+[y_i, x_j,y_j]\cdot y_i-[y_j,x_j,y_i]\cdot y_i+[y_i,x_j, y_j]\cdot y_i\nonumber\\
&+[x_i,y_j,y_j]\cdot y_i-[y_j, y_j,x_i]\cdot y_i\nonumber\\
=&-2[y_j,x_i, y_j]\cdot y_i+2[y_i, x_j,y_j]\cdot y_i.\nonumber
\end{align}
Thus
\[
\delta(p_{22})=0=2\sum_{j<i}\alpha_{ijji}(-[y_j,x_i, y_j]\cdot y_i+[y_i, x_j,y_j]\cdot y_i)
\]
and consequently
\[
\sum_{j<i}\alpha_{ijji}[y_j,x_i, y_j]\cdot y_i=0=\sum_{j<i}\alpha_{ijji}[y_i, x_j,y_j]\cdot y_i
\]
since the basis elements belong to different multilinear components for $j<i$.
We get $\alpha_{ijji}=0$ if $j<i$ that means $p_{22}=0$.

Similarly, because $\delta(p_{23})=0$ we have the sum
\[
\sum_{j<i<k}\left(\begin{array}{ll}\alpha_{ijik}(-2[y_j, x_i,y_i]\cdot y_k+[y_i, x_j,y_i]\cdot y_k+[y_i, x_i,y_j]\cdot y_k+[y_i, x_k,y_j]\cdot y_i)\\
+\alpha_{kjik}([y_k, x_j,y_i]\cdot y_k+[y_k, x_i,y_j]\cdot y_k-[y_j, x_k,y_i]\cdot y_k-[y_i, x_l,y_j]\cdot y_l)\\
+\alpha_{ijjk}(2[y_i, x_j,y_j]\cdot y_k-[y_j, x_i,y_j]\cdot y_k-[y_j, x_j,y_i]\cdot y_k-[y_j, x_k,y_j]\cdot y_i)\end{array}\right)\nonumber
\]
is zero, from which we deduce that each summand is zero; moreover,
\[
0=\sum_{j<i<k}\alpha_{ijik}[y_j, x_i,y_i]\cdot y_k=\sum_{j<i<k}\alpha_{kjik}[y_k, x_j,y_i]\cdot y_k=\sum_{j<i<k}\alpha_{ijjk}[y_i, x_j,y_j]\cdot y_k
\]
because the basis elements belong to different multilinear components. As above, we get $0=\alpha_{ijik}=\alpha_{kjik}=\alpha_{ijjk}$
for $j<i<k$ implying that $p_{23}=0$.

Again, writing $\delta(p_{24})$ as a linear combination of basis elements, we get
\begin{align}
0=&\sum_{j<i<k<l}\alpha_{kjil}\left(\begin{array}{ll}-[y_j, x_k,y_i]\cdot y_l+[y_k, x_j,y_i]\cdot y_l
-[y_j, x_i,y_k]\cdot y_l\\+[y_k, x_i,y_j]\cdot y_l-[y_i, x_l,y_j]\cdot y_k \end{array}\right)\nonumber\\
&+\sum_{j<i<k<l}\alpha_{ijkl}\left(\begin{array}{ll}-[y_j, x_i,y_k]\cdot y_l+[y_i, x_j,y_k]\cdot y_l-[y_j, x_k,y_i]\cdot y_l\\
+[y_i, x_k,y_j]\cdot y_l-[y_i, x_l,y_j]\cdot y_k+[y_j, x_l,y_i]\cdot y_k\end{array}\right)\nonumber
\end{align}
and we get $\alpha_{kjil}=\alpha_{ijkl}=0$ for $j<i<k<l$. Thus $p_{24}=0$, from which we get $p_2=0$.

\

\noindent\underline{Case 2} We study $(KB_3)^{\delta}$. Let
\[
p_3=\sum_{ j<i, \ j\leq k}\alpha_{ijkl}[x_i, x_j,x_k]\cdot y_l=p_{32}+p_{33}+p_{34}
\]
be a constant, where
\[
p_{32}=\sum_{j<i}\alpha_{ijii}[x_i, x_j,x_i]\cdot y_i+\alpha_{ijij}[x_i, x_j,x_i]\cdot y_j+\alpha_{ijji}[x_i, x_j,x_j]\cdot y_i+\alpha_{ijjj}[x_i, x_j,x_j]\cdot y_j,
\]
\[
p_{33}=\sum_{j<i<k}\left(\begin{array}{ll}\alpha_{ijik}[x_i, x_j,x_i]\cdot y_k+\alpha_{ijik}[x_i, x_j,x_i]\cdot y_k+\alpha_{ijik}[x_i, x_j,x_i]\cdot y_k\\
+\alpha_{ijkk}[x_i, x_j,x_k]\cdot y_k+\alpha_{ijki}[x_i, x_j,x_k]\cdot y_i+\alpha_{ijkj}[x_i, x_j,x_k]\cdot y_j\\
+\alpha_{kjik}[x_k, x_j,x_i]\cdot y_k+\alpha_{kjii}[x_k, x_j,x_i]\cdot y_i+\alpha_{kjij}[x_k, x_j,x_i]\cdot y_j\\
+\alpha_{ijjk}[x_i, x_j,x_j]\cdot y_k+\alpha_{kjji}[x_k, x_j,x_k]\cdot y_i+\alpha_{kiij}[x_k, x_i,x_i]\cdot y_j\end{array}\right),
\]
\[
p_{34}=\sum_{j<i<k<l}\left(\begin{array}{ll}\alpha_{ijkl}[x_i, x_j,x_k]\cdot y_l+\alpha_{ijlk}[x_i, x_j,x_l]\cdot y_k+\alpha_{kjli}[x_k, x_j,x_l]\cdot y_i\\
+\alpha_{kilj}[x_k, x_i,x_l]\cdot y_j+\alpha_{kjil}[x_k, x_j,x_i]\cdot y_l+\alpha_{ljik}[x_l, x_j,x_i]\cdot y_k\\
+\alpha_{ljki}[x_l, x_j,x_k]\cdot y_i+\alpha_{likj}[x_l, x_i,x_k]\cdot y_j\end{array}\right).
\]
As well as in the previous case, $\delta(p_{32})=\delta(p_{33})=\delta(p_{34})=0$. Let us study the case $\delta(p_{32})=0$. An easy computation shows the elements
\[
[x_i, x_j,x_i]\cdot y_i \ , \ [x_i, x_j,x_j]\cdot y_j \ , \ j<i
\]
are constants, then we get
\[
0=\sum_{j<i}(-\alpha_{ijij}+\alpha_{ijji})[x_i, x_j,x_j]\cdot x_i
\]
 and consequently $\alpha_{ijij}=\alpha_{ijji}$. Moreover, we get the elements of the form
\[
[x_i, x_j,x_i]\cdot y_j+[x_i, x_j,x_j]\cdot y_i \ , \ j<i
\]
are constants, too.

Notice also
\[
[x_i, x_j,x_k]\cdot y_k \ , \ [x_k, x_j,x_i]\cdot y_i \ , \ j<i<k.
\]
are constants, then $\delta(p_{33})=0$ may be rewritten as
\[
0=\sum_{j<i<k}\left(\begin{array}{ll}(\alpha_{ijik}-\alpha_{ijki}-\alpha_{kiij})[x_i, x_j,x_i]\cdot x_k\\
(-\alpha_{kjki}-\alpha_{kikj}+\alpha_{kjik})[x_k, x_j,x_i]\cdot x_k\\
(-\alpha_{ijkj}-\alpha_{kjij}+\alpha_{ijjk}+\alpha_{kjji})[x_i, x_j,x_j]\cdot x_k\end{array}\right).
\]
Hence we get the following system of equations for $j<i<k$:
\[\left\{
\begin{array}{ll}\alpha_{ijik}=\alpha_{ijki}+\alpha_{kiij}\\
\alpha_{kjki}=\alpha_{kjik}-\alpha_{kikj}\\
\alpha_{ijkj}=-\alpha_{kjij}+\alpha_{ijjk}+\alpha_{kjji}\end{array}\right..
\]
By the previous system, we get the next polynomials are constants, too.
\[
[x_i, x_j, x_i]\cdot  y_k+[x_i, x_j, x_k]\cdot  y_i\ , \ [x_i, x_j, x_i]\cdot  y_k+[x_k, x_i, x_i]\cdot  y_j
\]
\[
[x_k, x_j, x_k]\cdot  y_i+[x_k, x_j, x_i]\cdot  y_k \ , \ [x_k, x_j, x_k]\cdot  y_i-[x_k, x_i, x_k]\cdot  y_j
\]
\[
[x_i, x_j, x_k]\cdot  y_j-[x_k, x_j, x_i]\cdot  y_j\ , \ [x_i, x_j, x_k]\cdot  y_j+[x_i, x_j, x_j]\cdot  y_k
\]
\[
[x_i, x_j, x_k]\cdot  y_j+[x_k, x_j, x_j]\cdot  y_i \ , \ j<i<k.
\]
Let us rewrite now the relation $0=\delta(p_{34})$ as a linear combination of the basis elements.
\begin{align}
0=&\sum_{j<i<k<l}\left(\begin{array}{ll}\alpha_{ijkl}[x_i, x_j,x_k]\cdot x_l+\alpha_{ijlk}[x_i, x_j,x_l]\cdot x_k+\alpha_{kjli}[x_k, x_j,x_l]\cdot x_i\\
+\alpha_{kilj}[x_k, x_i,x_l]\cdot x_j+\alpha_{kjil}[x_k, x_j,x_i]\cdot x_l+\alpha_{ljik}[x_l, x_j,x_i]\cdot x_k\\
+\alpha_{ljki}[x_l, x_j,x_k]\cdot x_i+\alpha_{likj}[x_l, x_i,x_k]\cdot x_j\end{array}\right)\nonumber\\
=&\sum_{j<i<k<l}\left(\begin{array}{ll}(\alpha_{ijkl}-\alpha_{ijlk}+\alpha_{kilj}-\alpha_{ljik}+\alpha_{ljki})[x_i, x_j, x_k]\cdot  x_l\\
+(-\alpha_{kjli}-\alpha_{kilj}+\alpha_{kjil}+\alpha_{ljik}-\alpha_{ljki}-\alpha_{likj})[x_k, x_j, x_i]\cdot  x_l\end{array}\right).\nonumber
\end{align}
Hence we get
\[
\alpha_{ijkl}-\alpha_{ijlk}+\alpha_{kilj}-\alpha_{ljik}+\alpha_{ljki}=0
\]
\[
-\alpha_{kjli}-\alpha_{kilj}+\alpha_{kjil}+\alpha_{ljik}-\alpha_{ljki}-\alpha_{likj}=0
\]
for $j<i<k<l$, from which we have the following relations:
\[
\alpha_{ijkl}=\alpha_{ijlk}+\alpha_{kjli}-\alpha_{kjil}+\alpha_{likj}
\]
\[
\alpha_{kilj}=\alpha_{ljik}-\alpha_{ljki}-\alpha_{kjli}+\alpha_{kjil}-\alpha_{likj}.
\]
Substituting the relations above in the expression of $p_{34}$, 
we get the following constants:
\begin{align}
&[x_i, x_j,x_l]\cdot y_k+[x_i, x_j,x_k]\cdot y_l,\nonumber\\
&[x_l, x_j,x_i]\cdot y_k+[x_k, x_i,x_l]\cdot y_j,\nonumber\\
&[x_l, x_j,x_k]\cdot y_i-[x_k, x_i,x_l]\cdot y_j,\nonumber\\
&[x_k, x_j,x_l]\cdot y_i+[x_i, x_j,x_k]\cdot y_l-[x_k, x_i,x_l]\cdot y_j,\nonumber\\
&[x_k, x_j,x_i]\cdot y_l-[x_i, x_j,x_k]\cdot y_l+[x_k, x_i,x_l]\cdot y_j,\nonumber\\
&[x_l, x_i,x_k]\cdot y_j+[x_i, x_j,x_k]\cdot y_l-[x_k, x_i,x_l]\cdot y_j.\nonumber
\end{align}

\

\noindent\underline{Case 3} We study $(KB_4\oplus KB_5)^{\delta}$.
Let
\[
p_4=\sum_{j<i , \ k<l}\alpha_{ljki}[x_i, x_j,y_k]\cdot y_l \ \in KB_4 \ , \ p_5=\sum_{i\leq l , \ j\leq k}\beta_{ljki}[y_i, x_j,x_k]\cdot y_l\ \in KB_5
\]
such that $\delta(p_4+p_5)=0$. 
We point out $[y_i, x_i,x_i]\cdot y_i \ , [y_j, x_j,x_i]\cdot y_i \in(KB_4\oplus KB_5)^{\delta}$, $1\leq j<i\leq n$, since
\begin{align}
&\delta([y_i, x_i,x_i]\cdot y_i)=[x_i, x_i,x_i]\cdot y_i+[y_i, x_i,x_i]\cdot x_i=0,\nonumber\\
&\delta([y_j, x_j,x_i]\cdot y_i)=[x_j, x_j,x_i]\cdot y_i+[y_j, x_j,x_i]\cdot x_i=0.\nonumber
\end{align}
Thus we may assume $p_4+p_5=p_{452}+p_{453}+p_{454}$,
then
\[
\delta(p_4+p_5)=\delta(p_{452})+\delta(p_{453})+\delta(p_{454})=0,
\]
where
\[
p_{452}=\sum_{j<i}\left(\begin{array}{ll}\alpha_{ijji}[x_i, x_j,y_j]\cdot y_i+\beta_{ijii}[y_i, x_j,x_i]\cdot y_i+\beta_{ijji}[y_i, x_j,x_j]\cdot y_i\\
+\beta_{jjji}[y_j, x_j,x_j]\cdot y_i+\beta_{jjij}[y_j, x_j,x_i]\cdot y_j+\beta_{jiii}[y_j, x_i,x_i]\cdot y_i\\
+\beta_{jiij}[y_j, x_i,x_i]\cdot y_j\end{array}\right),
\]
\[
p_{453}=\sum_{j<i<k}\left(\begin{array}{ll}\alpha_{ijik}[x_i, x_j,y_i]\cdot y_k+\alpha_{kiji}[x_k, x_i,y_j]\cdot y_i+\alpha_{ijjk}[x_i, x_j,y_j]\cdot y_k\\
+\alpha_{kjji}[x_k, x_j,y_j]\cdot y_i+\alpha_{kijk}[x_k, x_i,y_j]\cdot y_k+\alpha_{kjik}[x_k, x_j,y_i]\cdot y_k\\
+\beta_{ijik}[y_i, x_j,x_i]\cdot y_k+\beta_{ijki}[y_i, x_j,x_k]\cdot y_i+\beta_{jiki}[y_j, x_i,x_k]\cdot y_i\\
+\beta_{jiik}[y_j, x_i,x_i]\cdot y_k+\beta_{ijjk}[y_i, x_j,x_j]\cdot y_k+\beta_{jjik}[y_j, x_j,x_i]\cdot y_k\\
+\beta_{jjki}[y_j, x_j,x_k]\cdot y_i+\beta_{jikj}[y_j, x_i,x_k]\cdot y_j+\beta_{ijkk}[y_i, x_j,x_k]\cdot y_k\\
+\beta_{kjik}[y_k, x_j,x_i]\cdot y_k+\beta_{jikk}[y_j, x_i,x_k]\cdot y_k+\beta_{jkki}[y_j, x_k,x_k]\cdot y_i\end{array}\right),
\]
\[
p_{454}=\sum_{j<i<k<l}\left(\begin{array}{ll}\alpha_{ijkl}[x_i, x_j,y_k]\cdot y_l+\alpha_{kjil}[x_k, x_j,y_i]\cdot y_l+\alpha_{ljik}[x_l, x_j,y_i]\cdot y_k\\
+\alpha_{kijl}[x_k, x_i,y_j]\cdot y_l+\alpha_{lijk}[x_l, x_i,y_j]\cdot y_k+\alpha_{lkji}[x_l, x_k,y_j]\cdot y_i\\
+\beta_{ijkl}[y_i, x_j,x_k]\cdot y_l+\beta_{kjil}[y_k, x_j,x_i]\cdot y_l+\beta_{ijlk}[y_i, x_j,x_l]\cdot y_k\\
+\beta_{jikl}[y_j, x_i,x_k]\cdot y_l+\beta_{jilk}[y_j, x_i,x_l]\cdot y_k+\beta_{jkli}[y_j, x_k,x_l]\cdot y_i\\
\end{array}\right).
\]
Rewriting $\delta(p_{452})\in KB_3$ as a linear combination of basis elements, we have
\[
0=\sum_{j<i}\left(\begin{array}{ll}(\alpha_{ijji}+2\beta_{ijji})[x_i, x_j,x_j]\cdot y_i+(-\alpha_{ijji}-2\beta_{jiij})[x_i, x_j,x_i]\cdot y_j\\
(\beta_{ijii}-\beta_{jiii})[x_i, x_j,x_i]\cdot y_i+(\beta_{jjji}-\beta_{jjij})[x_i, x_j,x_j]\cdot y_j\end{array}\right),
\]
then $\alpha_{ijji}=-2\beta_{jiij}$, $\beta_{ijji}=\beta_{jiij}$, $\beta_{ijii}=\beta_{jiii}$, $\beta_{jjji}=\beta_{jjij}$.
Thus, for $j<i$, the next are constants:
\begin{align}
&[y_j, x_i,x_i]\cdot y_j-2[x_i, x_j,y_j]\cdot y_i+[y_i, x_j,x_j]\cdot y_i,\nonumber\\
&[y_j, x_i,x_i]\cdot y_i+[y_i, x_j,x_i]\cdot y_i \ , \ [y_j, x_j,x_i]\cdot y_j+[y_j, x_j,x_j]\cdot y_i.\nonumber
\end{align}
A careful check of the coefficients of $[x_i, x_j,x_i]\cdot y_k$, $[x_i, x_j,x_k]\cdot y_i$, $[x_k, x_j,x_i]\cdot y_i$, $[x_k, x_i,x_i]\cdot y_j$, $[x_k, x_j,x_i]\cdot y_k$,
$[x_k, x_j,x_k]\cdot y_i$, $[x_i, x_j,x_k]\cdot y_k$, $[x_k, x_i,x_k]\cdot y_j$, $[x_i, x_j,x_j]\cdot y_k$, $[x_i, x_j,x_k]\cdot y_j$, $[x_k, x_j,x_j]\cdot y_i$,
$[x_k, x_j,x_i]\cdot y_j$, respectively, in the expression of $\delta(p_{453})=0$, gives us the following homogeneous system of equation:
\[
\begin{array}{ll}
0=\alpha_{ijik}+\beta_{ijik}-\beta_{jiik} & 0=-\alpha_{kijk}+\beta_{ijkk}-\beta_{jikk}-\beta_{kjik}\\
0=-\alpha_{ijik}-\alpha_{kiji}-\beta_{ijik}+2\beta_{ijki}-\beta_{jiki} & 0=-\alpha_{kijk}-\beta_{jkki}\\
0=\alpha_{kiji}+\beta_{ijik}-\beta_{ijki} & 0=\alpha_{ijjk}+\beta_{ijjk}\\
0=-\alpha_{kiji}-\beta_{jiki}+\beta_{jiik}& 0=-\alpha_{ijjk}-\beta_{jjik}+\beta_{jjki}-\beta_{jikj}\nonumber\\
0=\alpha_{kjik}+\alpha_{kijk}+2\beta_{kjik} & 0=\alpha_{kjji}+\beta_{ijjk}\nonumber\\
0=-\alpha_{kjik}-\beta_{jkki} & 0=-\alpha_{kjji}+\beta_{jjik}-\beta_{jjki}-\beta_{jikj}\end{array}
\]
We have the next relations:
\[
\alpha_{ijik}=\alpha_{kiji}=\beta_{ijki}-\beta_{jiki} \ , \ \beta_{ijik}=\beta_{jiki} \ , \ \beta_{jiik}=\beta_{ijki}
\]
\[
\alpha_{kjik}=\alpha_{kijk}=-\beta_{jkki}=-\beta_{kjik} \ , \ \beta_{ijkk}=\beta_{jikk}
\]
\[
\alpha_{ijjk}=\alpha_{kjji}=-\beta_{ijjk}=-\beta_{jikj} \ , \ \beta_{jjki}=\beta_{jjik}.
\]
Now substituting $\beta_{ijki}$, $\beta_{jiki}$, $\alpha_{kjik}$, $\beta_{ijkk}$, $\beta_{jjki}$, $\alpha_{ijjk}$, respectively,
in the expression of $p_{453}$, we obtain the following constants for $j<i<k$:
\begin{align}
&[y_i, x_j, x_k]\cdot  y_i+[x_i, x_j, y_i]\cdot  y_k+[x_k, x_i, y_j]\cdot  y_i +[y_j, x_i, x_i]\cdot  y_k\nonumber\\
&[y_j, x_i, x_k]\cdot  y_i-[x_i, x_j, y_i]\cdot  y_k-[x_k, x_i, y_j]\cdot  y_i +[y_i, x_j, x_i]\cdot  y_k\nonumber\\
&[x_k, x_j, y_i]\cdot  y_k+[x_k, x_i, y_j]\cdot  y_k -[y_j, x_k, x_k]\cdot  y_i -[y_k, x_j, x_i]\cdot  y_k\nonumber\\
&[y_i, x_j, x_k]\cdot  y_k+[y_j, x_i, x_k]\cdot  y_k \ , \ [y_j, x_j, x_k]\cdot  y_i+[y_j, x_j, x_i]\cdot  y_k \nonumber\\
&[x_i, x_j, y_j]\cdot  y_k+[x_k, x_j, y_j]\cdot  y_i -[y_i, x_j, x_j]\cdot  y_k-[y_j, x_i, x_k]\cdot  y_j\nonumber
\end{align}
Similar arguments for $\delta(p_{454})=0$ give out the following relations:
\[
\alpha_{ijkl}=\alpha_{lkji} \ , \ \alpha_{kijl}=\alpha_{ljik} \ , \ \beta_{ijkl}=\beta_{jilk} \ , \ \beta_{kjil}=\beta_{jkli}
\]
\[
\beta_{jikl}=\beta_{jilk}+\alpha_{lkji}-\alpha_{ljik}=\beta_{ijlk} \ , \ \alpha_{lijk}=-\beta_{jkli}-\alpha_{ljik}=\alpha_{kjil}
\]
and substituting $\alpha_{lkji}$, $\beta_{jkli}$, $\beta_{jilk}$, $\alpha_{ljik}$, in the expression of $p_{454}$,
\[
[x_l, x_k, y_j]\cdot  y_i +[x_i, x_j, y_k]\cdot  y_l+[y_j, x_i, x_k]\cdot  y_l+[y_i, x_j, x_l]\cdot  y_k
\]
\[
[y_j, x_k, x_l]\cdot  y_i+[y_k, x_j, x_i]\cdot  y_l-[x_l, x_i, y_j]\cdot  y_k-[x_k, x_j, y_i]\cdot  y_l 
\]
\[
[y_j, x_i, x_l]\cdot  y_k+[y_i, x_j, x_k]\cdot  y_l +[y_j, x_i, x_k]\cdot  y_l+[y_i, x_j, x_l]\cdot  y_k
\]
\[
[x_l, x_j, y_i]\cdot  y_k+[x_k, x_i, y_j]\cdot  y_l -[y_j, x_i, x_k]\cdot  y_l-[y_i, x_j, x_l]\cdot  y_k-[x_l, x_i, y_j]\cdot  y_k-[x_k, x_j, y_i]\cdot  y_l
\]
are constants, where $j<i<k<l$.

\

\noindent\underline{Case 4} Finally we have to show $(KB_6)^{\delta}=\{0\}$.
Let
\[
p_6=\sum_{i\leq l , \ k<l}\alpha_{ijkl}[y_i, x_j,y_k]\cdot y_l
\]
be a constant of $KB_6$. 
We have $p_6=p_{62}+p_{63}+p_{64}$, 
where
\[
p_{62}=\sum_{j<i}\alpha_{iiji}[y_i, x_i,y_j]\cdot y_i+\alpha_{ijji}[y_i, x_j,y_j]\cdot y_i+\alpha_{jiji}[y_j, x_i,y_j]\cdot y_i+\alpha_{jjji}[y_j, x_j,y_j]\cdot y_i,
\]
\[
p_{63}=\sum_{j<i<k}\left(\begin{array}{ll}
\alpha_{ikji}[y_i, x_k,y_j]\cdot y_i+\alpha_{kijk}[y_k, x_i,y_j]\cdot y_k+\alpha_{kjik}[y_k, x_j,y_i]\cdot y_k\\
+\alpha_{ikjk}[y_i, x_k,y_j]\cdot y_k+\alpha_{iijk}[y_i, x_i,y_j]\cdot y_k+\alpha_{ijjk}[y_i, x_j,y_j]\cdot y_k\\
+\alpha_{jkji}[y_j, x_k,y_j]\cdot y_i+\alpha_{jijk}[y_j, x_i,y_j]\cdot y_k+\alpha_{ijik}[y_i, x_j,y_i]\cdot y_k\\
+\alpha_{jkik}[y_j, x_k,y_i]\cdot y_k+\alpha_{jiik}[y_j, x_i,y_i]\cdot y_k+\alpha_{jjik}[y_j, x_j,y_i]\cdot y_k \end{array}\right),
\]
\[
p_{64}=\sum_{j<i<k<l}\left(\begin{array}{ll}
\alpha_{iljk}[y_i, x_l,y_j]\cdot y_k+\alpha_{ikjl}[y_i, x_k,y_j]\cdot y_l+\alpha_{kijl}[y_k, x_i,y_j]\cdot y_l\\
+\alpha_{kjil}[y_k, x_j,y_i]\cdot y_l+\alpha_{jlik}[y_j, x_l,y_i]\cdot y_k+\alpha_{jkil}[y_j, x_k,y_i]\cdot y_l\\
+\alpha_{jikl}[y_j, x_i,y_k]\cdot y_l+\alpha_{ijkl}[y_i, x_j,y_k]\cdot y_l\end{array}\right).
\]
Arguing as in the previous cases, we get all coefficients are 0 and we are done with the spanning elements. The final step proving that these spanning elements, all written as linear combinations of basis elements of the free metabelian Poisson algebra, is straightforward.
\end{proof}

\end{document}